\documentclass[12pt,leqno,a4paper]{amsart}
\usepackage{amsmath,amsthm,amsfonts,amssymb,latexsym, mathabx}
\usepackage{enumerate}

\usepackage{color, colortbl, xcolor, makecell}

\newcommand{\FF}{{\mathbb{F}}}

\newcommand{\Q}{{\mathbb{Q}}}

\newcommand{\bG}{{\mathbf{G}}}
\newcommand{\Irr}{{\operatorname{Irr}}}

\newcommand{\Out}{{\operatorname{Out}}}
\newcommand{\out}{{\operatorname{Out}}}

\newcommand{\SL}{{\operatorname{SL}}}

\newcommand{\PSL}{{\operatorname{PSL}}}

\newcommand{\POmega}{{\operatorname{P\Omega}}}

\newcommand{\alt}{\mathfrak{A}}
\newcommand{\sym}{\mathfrak{S}}

\newcommand{\tw}[1]{{}^#1\!}

\def\syl#1#2{{\rm Syl}_#1(#2)}
\def\nor{\triangleleft\,}

\def\oh#1#2{{\bf O}_{#1}(#2)}

\def\zent#1{{\bf Z}(#1)}
\def\irr#1{{\rm Irr}(#1)}
\def\norm#1#2{{\bf N}_{#1}(#2)}
\def\cent#1#2{{\bf C}_{#1}(#2)}
\def\ker#1{{\rm Ker}(#1)}
\def\irrp#1#2{{\rm Irr}_{#1'}(#2)}

\def\sbs{\subseteq}

\newtheorem{thm}{Theorem}[section]
\newtheorem{lem}[thm]{Lemma}

\newtheorem{prop}[thm]{Proposition}
\newtheorem{question}[thm]{Question}

\newtheorem*{thm*}{Theorem $\heartsuit$}

\newtheorem*{conjA}{Conjecture A}
\newtheorem*{conjB}{Conjecture B}

\newtheorem*{conjC}{Conjecture C}
\newtheorem*{thmD}{Theorem D}
\newtheorem*{thmE}{Theorem E}

\theoremstyle{remark}
\newtheorem{rem}[thm]{Remark}

\newcommand\wt[1]{\widetilde{#1}}

\usepackage{array}
\newcolumntype{?}{!{\vrule width 1pt}}

\newcommand{\GL}{\operatorname{GL}}

\newcommand{\PSp}{\operatorname{PSp}}

\newcommand{\aut}{\operatorname{Aut}}
\newcommand\type[1]{\operatorname{#1}}

\newcommand{\gal}{\mathcal{G}}

\begin{document}

\title[Principal Blocks for different Primes, I]{Principal Blocks for different Primes, I}

\author{Gabriel Navarro}
\address[G. Navarro]{Departament de Matem\`atiques, Universitat de Val\`encia, 46100 Burjassot,
       Val\`encia, Spain.}
\email{gabriel@uv.es}

\author{Noelia Rizo}
\address[N. Rizo]{Departamento de Matem\'aticas, Universidad de Oviedo, 33007, Oviedo, Spain}
\email{rizonoelia@uniovi.es}

\author{A. A. Schaeffer Fry}
\address[A. A. Schaeffer Fry]{Deptartment of  Mathematics and Statistics, Metropolitan State University of Denver, Denver, CO 80217, USA}
\email{aschaef6@msudenver.edu}

\thanks{The first and second authors are partially supported by Grant PID2019-103854GB-I00 funded by MCIN/AEI/10.13039/501100011033. The second author also acknowledges support by Generalitat Valenciana AICO/2020/298 and Grant PID2020-118193GA-I00 funded by MCIN/AEI/10.13039/501100011033.  The third author is partially supported by a grant from the National Science Foundation, Award No. DMS-2100912.}



\begin{abstract}  
We propose new conjectures about the relationship between the principal blocks of finite groups
for different primes {and establish evidence for these conjectures}. \end{abstract}

\maketitle

\centerline{\sl To Pham Huu Tiep, on his 60th birthday}
\bigskip

\section{Introduction}
If $p$ and $q$ are different primes and $G$ is a finite group,
it is not generally reasonable to expect meaningful interactions between
the $p$-representation theory of $G$ and its $q$-representation theory, outside of the case of solvable groups.
But, there are some  exceptions.

\medskip

In 1997 (\cite{NW}), W. Willems and the first author asked when the irreducible
complex characters in a Brauer $p$-block $B_p$ coincide
with those in a $q$-block $B_q$. This was later characterized
 for  principal blocks $B_p(G)$ in \cite{BNOT}, using the Classification of Finite Simple Groups (CFSG): we have that 
 $\irr{B_p(G)}=\irr{B_q(G)}$  if and only if
 $p$ and $q$ do not divide $|G|$. In other words,
the subset $\irr{B_p(G)}$ of the irreducible complex characters of $G$ in the
principal $p$-block of $G$
determines the prime $p$, whenever $p$ divides $|G|$.

\medskip

In 2008, C. Bessenrodt and J. Zhang \cite{BZ} studied the opposite case, the {\sl trivial intersection} 
case, and proved, again using the CFSG,
 that $G$ is nilpotent if and only if the trivial character of $G$, $1_G$, is the only irreducible complex character lying in $\irr{B_p(G)} \cap \irr{B_q(G)}$
for all primes $p \ne q$ dividing the order of $G$. This result  led the authors in \cite{LWXZ} to recently conjecture that
if $\irr{B_p(G)} \cap \irr{B_q(G)}=\{ 1_G\}$, then there is a Sylow $p$-subgroup
$P$ of $G$ and a Sylow $q$-subgroup $Q$ of $G$ whose elements commute.
In \cite{LWXZ}, this conjecture is reduced to almost-simple groups, and it is proved if one of the 
primes is 2. The general case of this interesting and deep problem seems out of reach.

\medskip 
There is a third recently-observed interaction between principal blocks and different primes,  which interestingly has the same
conclusion as the previous one. The
so called {\sl Brauer's height zero conjecture for two primes} was proposed in \cite{MN20}:
if $G$ is a finite group, then the elements of a Sylow $p$-subgroup of $G$ commute with 
the elements of some Sylow $q$-subgroup of $G$
if and only if $p$ does not divide the degrees of the characters in $B_q(G)$ and $q$ does not divide
the degrees of the characters in $B_p(G)$. This has been recently proved in \cite{LWWZ}.

\medskip

In this paper, we propose a strengthened version of these conjectures, which we can prove in certain cases.   (We also fix a gap in the  reduction in \cite{LWXZ} which, after informing the authors,   was corrected in \cite{LWXZ2}.)
There is a new idea behind our conjectures:   it should be possible to replace $\irr{B_p(G)}$ by
the much smaller subset $\irrp{p}{B_p(G)}$ of the characters in the 
principal $p$-block of $G$ of degree not divisible by $p$, while still obtaining the corresponding conclusions.
\medskip

Accordingly, we conjecture the following.
\medskip
\begin{conjA}
Let $G$ be a finite group and let $p$ and $q$ be different primes.
If $$\irrp{p}{B_p(G)} \cap  \irrp{q}{B_q(G)}=\{1_G\},$$
then there are a Sylow $p$-subgroup $P$ of $G$
and a Sylow $q$-subgroup $Q$ of $G$ such that $xy=yx$ for all
$x \in P$ and $y \in Q$.
\end{conjA}

\begin{conjB}
Let $G$ be a finite group and let $p$ and $q$ be primes dividing the order of $G$.
If $\irrp{p}{B_p(G)}=\irrp{q}{B_q(G)}$, then $p=q$.
\end{conjB}

\begin{conjC}
Let $G$ be a finite group, and let $p$ and $q$ be different primes.
Then $q$ does not divide $\chi(1)$ for all $\chi \in \irrp{p}{B_p(G)}$
and $p$ does not divide $\chi(1)$ for all $\chi \in \irrp{q}{B_q(G)}$
if and only if there are a Sylow $p$-subgroup $P$ of $G$
and a Sylow $q$-subgroup $Q$ of $G$ such that $xy=yx$ for all
$x \in P$ and $y \in Q$.
\end{conjC}

At the end of the paper, we shall also discuss a further strengthening of our conjectures using Galois automorphisms. In the present paper, we mainly focus on Conjecture A, and prove the following.

\begin{thmD}
\begin{enumerate}[(a)]

\item
Conjecture A is true for all finite groups if it is true for almost--simple groups.
\item
Conjecture A is true if $p=2$.
\end{enumerate}
 \end{thmD}
\medskip
 {We prove part (a) of Theorem D in Section \ref{sectionreduction} below, and complete part (b) in Section \ref{sectionsimples}.} As the reader will see, the most difficult part in this paper is to prove Conjecture A for $p=2$ and almost simple groups.  {We remark that in the proof for simple groups, it is often clear that there exist nontrivial characters lying in both principal blocks, so the main step in these cases is to verify that such characters may be found to also have degree prime to $p$ and $q$, a step not needed in  \cite{LWXZ}.}
There is a rather surprising fact that makes Conjecture A, as well as the main conjecture of \cite{LWXZ}, difficult to prove.
Perhaps it is worth mentioning it now: if $N\nor G$, $G/N$ has order not divisible by $p$,
and $G=N\cent GP$, where $P \in \syl pG$, it is well known that $\theta \in \irr{B_p(N)}$ has a canonical extension
$\theta_p \in \irr{B_p(G)}$  (using the Alperin-Dade theory of isomorphic blocks).
If $q \ne p$ is also a prime number, however, under the same circumstances, it might very well occur that $\theta_p \ne \theta_q$. So $\theta$
can have in this case {\sl two} canonical extensions to $G$. (See Remark \ref{rem:oddp1complication}.)

\medskip

The following is the second main result of this paper.

 \begin{thmE}
If $G$ is $p$-solvable and $q$-solvable, then Conjectures B and C are true.
\end{thmE}

\medskip

{We prove Theorem E in Section \ref{sec:pqsolvable}.} In 
a subsequent paper (\cite{NRS2}),  we shall  focus more on Conjectures B and C providing solid evidence towards their validity.

\section{A reduction for Conjecture A}\label{sectionreduction}
In this section, we prove Theorem D, assuming the following
result on almost simple groups, which will be proved in Section \ref{sectionsimples}.
Our reduction is slightly different from the reduction in \cite{LWXZ}
(which uses the so called $p^*$-theory and \cite{W86} ).

As usual, if $H,K$ are subgroups of $G$, then $[H,K]$ is the subgroup generated by the commutators
$[h,k]$, for $h \in H$ and $k \in K$. 

\medskip

We start with two elementary lemmas.

\begin{lem}\label{th}
Let $G$ be a finite group
and let $p$ be a prime. Then $\irrp{p}{B_p(G)}=1_G$
if and only if $G$ is a $p'$-group.
\end{lem}

\begin{proof}
The if direction is obvious. Suppose now that $\irrp{p}{B_p(G)}=1_G$. By the main result of \cite{IS}, we know that
$G$ has a normal $p$-complement $K$.
Suppose that $K<G$, and let $1 \ne \lambda \in \irr{G/K}$ linear. Since $K\leq {\rm Ker}(\lambda)$, we have that $\lambda\in{\rm Irr}_{p'}(B_p(G))$, a contradiction.
\end{proof}

As usual, $\irrp{p}G$ denotes the subset of $\irr G$ consisting of irreducible characters of degree not divisible by $p$.
 
 \begin{lem}\label{tr}
 Let $G$ be a finite group,
and let $p$ be a prime. If $\irrp{p}G=\{1\}$, then $G=1$.
 \end{lem}
 \begin{proof}
 Since $|G|=1 + \sum_{1_G \ne \chi \in \irr G} \chi(1)^2$,
 we deduce that $|G|$ is not divisible by $p$.
 Thus $\irr G=1$ and $G=1$.
  \end{proof}

 As we have mentioned in the introduction,
 the reduction theorem in \cite{LWXZ} contains a gap.
 Indeed, in the first step of the proof of their Theorem 1.4,
 the authors conclude that $G$ has a unique minimal normal subgroup
 by using an argument which is not correct. 
 Suppose that $\pi=\{p,q\}$, 
 that $N$ and $M$ are distinct normal subgroups
 of a finite group $G$ and that $G/N$ and $G/M$ contain nilpotent Hall $\pi$-subgroups. 
 The authors argue that $G$ also contains a nilpotent Hall
 $\pi$-subgroup, because $G$ is isomorphic to a subgroup of $\hat G=(G/N) \times (G/M)$,
 which does have a nilpotent Hall $\pi$-subgroup. It is true, as they say, 
 that every $\{p,q\}$-subgroup of $G$ is contained in a Hall $\pi$-subgroup
 of $\hat G$, by using Wielandt's theorem, but this does not prove that $G$
 has a Hall $\pi$-subgroup. It is false that arbitrary
 subgroups $H$ of $\hat G$  possess
 Hall $\pi$-subgroups: it is enough to consider $\hat G={\rm PSL}_2(16)$, and $H={\sf A}_5$,
 for $\pi=\{3,5\}$. Nevertheless, the conclusion of Step (1) of \cite[Theorem 1.4]{LWXZ}
 is correct. However, the argument needs the CFSG and no less than
 the complete classification
 of the Hall $\pi$-subgroups of finite simple groups. The following is \cite[Corollary 8]{RV},
 which we shall also use.
 
 \begin{lem}\label{vdovin}
 Suppose that $G$ is a finite group. Suppose that
 $N, M \nor G$. If $G/N$ and $G/M$ have nilpotent
 Hall $\pi$-subgroups, then so does $G/(N\cap M)$.
 \end{lem}

We will also make use of the Alperin-Dade's theory of isomorphic blocks.

\begin{thm}\label{isomblocks}
Suppose that $N$ is a normal subgroup of $G$, with $G/N$ a $p'$-group.
Let $P \in \syl p G$ and assume that $G=N\cent GP$. Then restriction of characters defines
a natural bijection between ${\rm Irr}(B_p(G))$ and ${\rm Irr}(B_p(N))$.

\end{thm}

\begin{proof}
The case where $G/N$ is solvable was proved in \cite{Alp76} and the general case
in \cite{Dad77}. 
\end{proof}

We prove next that Conjecture A is true provided the following question on almost simple groups holds true. As a consequence, we 
 obtain Theorem D (a).

\begin{question}\label{questionalmostsimple}
Let $p$ and $q$ be different prime numbers.
Suppose that
 $G$ is almost simple with socle $S$, a non-abelian simple group. 
 Assume that  $pq\mid |S|$ and
 that  $G=S\cent G P=S\cent G Q$, for every $P\in{\rm Syl}_p(S)$ and $Q\in{\rm Syl}_q(S)$.
 If $\irrp{p}{B_p(G)} \cap  \irrp{q}{B_q(G)}=1$,
then there is a Sylow $p$-subgroup $P_0$
of $G$ and a Sylow $q$-subgroup $Q_0$ of $G$ such that $[P_0,Q_0]=1$.
\end{question}

 Notice that Question \ref{questionalmostsimple} is slightly weaker than Conjecture A.
 
 \medskip
 In the final step of the main theorem of this section, we use some elementary group theory.
 
 \begin{lem}\label{expla}
 Suppose that $G$ is a finite group, $S \nor G$, $G=S\cent G{P_1}$, where $P_1 \in \syl pS$.
 Suppose that $|S|$ has order not divisible by $q$ for some prime $q$. If $G/S$ has a nilpotent Hall
 $\{p,q\}$-subgroup, then $G$ has a nilpotent Hall $\{p,q\}$-subgroup.
  \end{lem}
  
  \begin{proof}
  We use induction on $|G|$.  By induction, we may assume that  $G=P_1\cent G{P_1}$ and that $S=P_1\cent S{P_1}$.
  In particular, $P_1$ and $\cent G{P_1}$ are normal in $G$. 
  By the Schur-Zassenhaus theorem, we can write $S=P_1 \times X$, for some normal $p'$-subgroup $X$ of $G$.
  If $X>1$, we can apply induction to $G/X$, to conclude that $G/X$ has a nilpotent Hall $\{p,q\}$-subgroup.
  Since $X$ is not divisible by $p$ and $q$, again by Schur-Zassenhaus, we conclude that $G$ has a nilpotent Hall $\{p,q\}$-subgroup.
  Therefore, we may assume that $\cent S{P_1}=\zent{P_1}$.
  By hypothesis, $G/S$ and therefore $\cent G{P_1}/\zent{P_1}$ have a nilpotent 
  Hall $\{p,q\}$-subgroup. Let $Q \in \syl q{\cent G{P_1}}$, and let $P_2 \in \syl p{\cent G{P_1}}$
  such that $[P_2, Q]\sbs \zent{P_1}$. Since $[P_1, Q]=1$, we have that $[P_2,Q,Q]=1$. By coprime action (see \cite[Lemma 4.29]{Is2}),
  we conclude that $[P_2, Q]=1$.
  Notice that $P=P_1P_2 \in \syl pG$ and that $[P,Q]=1$.  \end{proof}
 
\begin{thm}\label{thmreduction}
Let $p$ and $q$ be primes such that Question \ref{questionalmostsimple} holds, and let $G$ be a finite group.
 Suppose that $${\rm Irr}_{p'}(B_p(G))\cap {\rm Irr}_{q'}(B_q(G))=\{1_G\} \, .$$ Then there exist $P\in{\rm Syl}_p(G)$ and $Q\in{\rm Syl}_q(G)$ such that  $[P,Q]=1$. 
\end{thm}

\begin{proof}

Let $\pi=\{p,q\}$. We want to show that $G$ has a nilpotent Hall $\pi$-subgroup. We proceed by induction on $|G|$. 
\medskip

\textit{Step 0. If $1<N$ is normal in $G$, then $G/N$ has a nilpotent Hall $\pi$-subgroup.}
\smallskip

We have that $${\rm Irr}_{p'}(B_p(G/N))\cap{\rm Irr}_{q'}(B_q(G/N))\subseteq{\rm Irr}_{p'}(B_p(G))\cap {\rm Irr}_{q'}(B_q(G))=\{1_G\},$$
and we apply the inductive hypothesis.

\medskip

\textit{Step 1. $G$ has exactly
 one minimal normal subgroup $K$.}
 \smallskip
 
 This follows from Lemma \ref{vdovin} and Step 0.
 
 \medskip
 
\textit{Step 2. If $1<N$ is a normal subgroup of $G$, then $p$ divides $|N|$ or $q$ divides $|N|$.}
\smallskip

Otherwise, by Step 0, we have that $G/N$
has a nilpotent Hall $\pi$-subgroup $H/N$.  Since $|N|$ is coprime to $pq$ then there is $H_1$ with $H=H_1N$ and $H_1\cap N=1$,
by the Schur-Zassenhaus theorem.
 Then $H_1\cong H/N$ is a nilpotent Hall $\{p,q\}$-subgroup of $G$.
 
 \medskip
 
\textit{Step 3. If $1<N$ is
a normal subgroup of $G$,
then ${\rm Irr}_{p'}(B_q(G/N))={\rm Irr}(B_q(G/N))$ and  ${\rm Irr}_{q'}(B_p(G/N))={\rm Irr}(B_p(G/N))$.}
\smallskip

By Step 0,
we have that there are $P\in{\rm Syl}_{p}(G)$ and $Q\in{\rm Syl}_q(G)$ with $[PN/N,QN/N]=1$. By \cite[Theorem 4.1]{MN20},
 we have that ${\rm Irr}_{p'}(B_q(G/N))={\rm Irr}(B_q(G/N))$ and  ${\rm Irr}_{q'}(B_p(G/N))={\rm Irr}(B_p(G/N))$, as wanted.

\medskip

\textit{Step 4. If $1<N\lhd G$, $P_1\in{\rm Syl}_p(N)$ with $\cent G {P_1}\subseteq N$, 
then $|G/N|$ is not divisible by $q$, $G=N\cent GQ$ for some Sylow $q$-subgroup of $G$, and  ${\rm Irr}_{p'}(B_p(N))\cap{\rm Irr}_{q'}(B_q(N))=\{1_N\}$.}
\smallskip

By the Third Main Theorem (see \cite[Theorem 6.7]{nbook}) and \cite[Lemma 3.1]{NT12},
 we have that $B_p(G)$ is the only block of $G$ covering $B_p(N)$. By \cite[Theorem 9.2]{nbook},
 we have that ${\rm Irr}(G/N) \subseteq {\rm Irr}(B_p(G))$. Then ${\rm Irr}(B_q(G/N)) \subseteq {\rm Irr}(B_p(G))$.
Since $G/N$ is a $p'$-group, we have that  
${\rm Irr}_{q'}(B_q(G/N))\subseteq {\rm Irr}_{p'} (B_p(G))\cap{\rm Irr}_{q'}(B_q(G)))=\{1_G\}$.
Hence
$G/N$ is a $q'$-group by Lemma \ref{th}.

Let $Q\in{\rm Syl}_{q}(G)$. Then $Q\subseteq N$, $G=N\norm G Q$ and $H=N\cent G Q\lhd G$. By the previous argument,
notice that $B_q(G)$ is the only $q$-block covering $B_q(H)$ and $\irr{G/H} \sbs {\rm Irr}_{q'}(B_q(G))$. 
Let $P_2 \in \syl pH$ containing $P_1$.
Since $\cent G{P_2} \sbs
\cent G {P_1} \subseteq N\subseteq H$, again we have that
$B_p(G)$ is the only $p$-block covering $B_p(H)$, and $\irr{G/H} \sbs \irr{B_p(G)}$. 
Then $${\rm Irr}_{p'}(G/H) \subseteq{\rm Irr}_{p'}(B_p(G))\cap{\rm Irr}_{q'}(B_q(G))=\{1_G\}$$ and $G=H$,
by Lemma \ref{tr}.

 Now, by Theorem \ref{isomblocks}  we have that restriction defines a bijection $${\rm Irr}(B_q(G))\rightarrow{\rm Irr}(B_q(N)).$$ Take $\theta\in{\rm Irr}_{p'}(B_p(N))\cap{\rm Irr}_{q'}(B_q(N))$, and let $\chi\in{\rm Irr}_{q'}(B_q(G))$
 be extending $\theta$. Since $B_p(G)$ is the only $p$-block covering $B_p(N)$ we have $\chi\in{\rm Irr}_{p'}(B_p(G))\cap{\rm Irr}_{q'}(B_q(G))$ and hence $\chi=1_G$ and therefore $\theta=1_N$, as desired.

\medskip

\textit{Step 5. Let $N$ be a proper normal subgroup of $G$ and let $P\in{\rm Syl}_p(N)$ and $Q\in{\rm Syl}_q(N)$. Then $\cent G P\not\subseteq N$ and $\cent G Q\not\subseteq N$. }
\smallskip

We prove for instance
that $\cent G P\not\subseteq N$. Suppose the contrary and let $M$ be a maximal normal subgroup of $G$ with $N\subseteq M$. Let $R\in{\rm Syl}_p(M)$ with $P\subseteq R$. Then $\cent G R\subseteq\cent G P\subseteq N\subseteq M$ and by Step 4 we have $q\nmid |G/M|$ and ${\rm Irr}_{p'}(B_p(M))\cap {\rm Irr}_{q'}(B_q(M))=\{1_M\}$. By induction there exists $S\in{\rm Syl}_p(M)$ and $T\in{\rm Syl}_q(M)$ with $[S,T]=1$. Since $q\nmid |G/M|$ we have $T\in{\rm Syl}_q(G)$. If $\cent G T\subseteq M$, by Step 4 (changing $p$ by $q$) we have that $p\nmid |G/M|$ and then $S\in{\rm Syl}_p(G)$ and we are done.  Hence we may assume $\cent G T\not\subseteq M$.

By the Frattini Argument we have $G=M\norm G T$ and hence $M\cent G T$ is normal in $G$. By the maximality of $M$ we have $G=M\cent G T$. Taking $p$-parts we have 

$$|G|_p=|M\cent G T|_p=\frac{|M|_p|\cent G T|_p}{|\cent M T|_p}.$$

Since $S\subseteq\cent M T$, we have that $|M|_p=|S|=|\cent M T|_p$ and hence

$$|G|_p= |\cent G T|_p.$$ Let $W\in{\rm Syl}_p(\cent G T)$, then  $W\in{\rm Syl}_p(G)$ and $[W,T]=1$, so we are done. Hence $\cent G P\not\subseteq N$. Analogously we prove $\cent G Q\not\subseteq N$ .

\medskip

\textit{Step 6. Let $N$ be a normal subgroup of $G$ and let $P\in{\rm Syl}_p(N)$ and $Q\in{\rm Syl}_q(N)$. Then $N\cent G P=G=N\cent G Q$. }
\smallskip

Take $H=N\cent G P$. Then by the Frattini Argument $H$ is normal in $G=N\norm G P$. Take $P_0\in{\rm Syl}_p(H)$ containing $P$. Then $\cent G {P_0}\subseteq\cent G P\subseteq H$. By Step 5, $H=G$. The same reasoning shows the equality $G=N\cent G Q$. 

\medskip

\textit{Step 7. The Fitting subgroup of $G$, ${\rm\textbf{F}}(G)$, is trivial. }
\smallskip

Otherwise,
let $U$ be a minimal normal abelian subgroup of $G$. By Step 2, we may assume that $U$ is
a $p$-group. By Step 6,
we have that $G=U\cent GU=\cent GU$, and we conclude that $U$
is central in $G$.   Since $U>1$, by Step 0,
there exist $Q\in{\rm Syl}_q(G)$ and $P\in{\rm Syl}_p(G)$ with $[QU/U,P/U]=1$. Then $Q$ acts trivially on $P/U$ and $U$,
and it follows that $Q$ acts trivially on $P$, by coprime action. 

\medskip

\textit{Step 8. The unique minimal normal subgroup $K$ of $G$
is a non-abelian simple group. So $G$ is almost simple with socle $S=K$. Moreover $pq\mid |S|$ and $G=S\cent G P=S\cent G Q$, for every $P\in{\rm Syl}_p(S)$ and $Q\in{\rm Syl}_q(S)$.}

By Step 7,  we know that $K=S_1\times\cdots\times S_t$, where $S=S_1$ is a non-abelian simple group and $S_i=S^{g_i}$ for some $g_i\in G$, for $i=1,\ldots,t$. Suppose that $t>1$, so in particular $K<G$. By Step 2,
 we have that $p$ divides $|K|$ or $q$ divides $|K|$. We may assume without loss of generality that $p$ divides $|K|$. Let
  $P\in{\rm Syl}_p(K)$ and let $P_1=P\cap S_1
  \in \syl p{S_1}$. By Step 6, we have that $G=K\cent G P$. Now, if  $g\in G$,
  then $g=xk$ for some $x\in \cent GP$ and $k\in K$, and we have $P_1^g=P_1^k \sbs S_1^k=S_1$ for some $k\in K$.
  Hence $P_1^g\leq S_1^g\cap S_1$. Since $P_1$ is non-trivial,
   we conclude that $S_1^g= S_1$ and $S_1$ is normal in $G$, contradicting the minimality of $N$. Hence $t=1$ and $K=S$ is non-abelian simple as wanted. Since $S$ is the unique minimal normal subgroup,
   we have that  $\cent G S=1$ and $G$ is almost simple with socle $S$.

 By Step 6,  we know that $p$ or $q$ divide $|S|$. We may assume that $p$ divides $|S|$, so we need to prove that $q\mid |S|$.
 Otherwise, we apply Lemma \ref{expla}.

Now the conclusion of the theorem holds by Question \ref{questionalmostsimple}.
\end{proof}

As a consequence of Theorem \ref{thmreduction} and Theorem \ref{almostsimple} (which we prove in the
next section),  we obtain Theorem D.

\medskip

\section{Almost Simple Groups}\label{sectionsimples}
In this section, we prove the following: 

\begin{thm}\label{almostsimple}
Let $p$ be an odd prime.
Suppose that
 $A$ is almost simple with socle $S$, a non-abelian simple group. 
 Assume that  $p\mid |S|$ and
 that  $A=S\cent A P=S\cent A Q$, for every $P\in{\rm Syl}_p(S)$ and $Q\in{\rm Syl}_2(S)$.
 If $\irrp{p}{B_p(A)} \cap  \irrp{2}{B_2(A)}=1$,
then there are a Sylow $p$-subgroup $P_0$
of $A$ and a Sylow $2$-subgroup $Q_0$ of $A$ such that $[P_0,Q_0]=1$.
\end{thm}

 We begin by restating Question \ref{questionalmostsimple}:

\begin{question}\label{almostsimpleLie}
Let $p$ and $q$ be different primes.
Suppose that
 $A$ is almost simple with socle $S$, where $S$ is a simple group with $pq\mid |S|$. 
 Assume 
 that  $A=S\cent A P=S\cent A Q$, for every $P\in{\rm Syl}_p(S)$ and $Q\in{\rm Syl}_q(S)$.
 If $\irrp{p}{B_p(A)} \cap  \irrp{q}{B_q(A)}=1$,
then there is a Sylow $p$-subgroup $P_0$
of $A$ and a Sylow $q$-subgroup $Q_0$ of $A$ such that $xy=yx$ for all $x\in P_0$ and all $y\in Q_0$.
\end{question}

Note the change of notation: we have used $A$, rather than $G$, to denote the almost simple group.  Throughout this section, we find it convenient to use $G$ to denote a quasisimple group satisfying $S=G/\zent{G}$, and hence we will use the notation $A$ for the almost simple group under consideration, rather than $G$.
When $A\neq S$, the following will be useful, which comes from \cite[Theorem A]{Gross82} and \cite[Theorem 10]{Glauberman}.
 \begin{thm}\label{A:Sp'}

 Let $p$ be a prime, and suppose $S$ is a finite nonabelian simple group with order divisible by $p$ and let $P\in \syl{p}{S}$. Write $C:=\cent{\aut(S)}{P}$.
 \begin{enumerate}
 \item (Gross) If $p\neq 2$, then $C=C_1\times Z(P)$, where $p\nmid |C_1|$ and we view $P\leq S\leq \aut(S)$ as inner automorphisms. 
 \item (Glauberman) If $p=2$, then $C=C_1C_2$ with $C_1\lhd C$, $2\nmid|C_1|$, and $C_2$ is the direct product of $Z(P)$ and an elementary abelian $2$-group.
 \end{enumerate}
\end{thm}

\begin{prop}\label{prop:sporadicexceptschur}
Question \ref{almostsimpleLie} holds (for any pair of primes) when $S$ is an alternating group $\alt_n$ with $n\leq 7$, a sporadic group, a group of Lie type with exceptional Schur multiplier, or the Tits group $\tw{2}\type{F}_4(2)'$.
\end{prop}
\begin{proof}
This can be seen using GAP and the GAP Character Table Library. Note that the groups with exceptional Schur multiplier can be found, e.g., in \cite[Table 6.1.3]{GLS3}.
\end{proof}

We also note that, in the context of Question \ref{almostsimpleLie}, we will often write $\pi:=\{p,q\}$ for the pair of primes under consideration. 

\subsection{Groups of Lie type}
Let $\bG$ be a connected reductive group defined over $\overline{\FF}_\ell$, where $\ell$ is some prime.  Let $F\colon \bG\rightarrow\bG$ be a Steinberg morphism endowing $\bG$ with an $\FF_{\ell^f}$-rational structure, for some power $\ell^f$ of $\ell$.  Then by a group of Lie type, we mean a finite group $G=\bG^F$, the set of fixed points of $\bG$ under $F$.  In this situation, we say that $G$ is defined in characteristic $\ell$.  

In particular, if $\bG$ is simple of simply connected type, then $G=\bG^F$ is  quasi-simple and $S=G/\zent{G}$ is a simple group of Lie type. In this case, we may let $\iota\colon\bG\hookrightarrow \wt{\bG}$ be a regular embedding as in \cite[Proposition 1.7.5]{GM20}  or \cite[(15.1)]{CE04}, 
and write $\wt{G}:=\wt{\bG}^F$.  Then $\zent{\wt{\bG}}$ is connected, $G\lhd \wt{G}$ with $\wt{G}/G$ abelian, and $S\cong G\zent{\wt{G}}/\zent{\wt{G}}$.   Write $\wt{S}:=\wt{G}/\zent{\wt{G}}$.  Then $\wt{S}/S\cong \wt{G}/G$ induces the group of diagonal automorphisms on $S$.  For a convenient choice $D$ of graph and field automorphisms, we have $\aut(S)=\wt{S}\rtimes D$.

The sets $\Irr(G)$ and $\Irr(\wt{G})$ are partitioned into rational Lusztig series $\mathcal{E}(G, s)$, resp. $\mathcal{E}(\wt{G}, s)$, indexed by $G^\ast$ (resp. $\wt{G}^\ast$)-conjugacy classes of semisimple elements $s$.   Here $G^\ast$ and $\wt{G}^\ast$ are the fixed points under $F$ of dual groups $\bG^\ast$ and $ \wt{\bG}^\ast$ of $\bG$ and $\wt{\bG}$, respectively.  The map $\iota$ further induces a surjection $\iota^\ast\colon \wt{\bG}^\ast\rightarrow\bG^\ast$.

In particular, $\mathcal{E}(\wt{G}, s)$ contains a unique so-called semisimple character $\chi_s$, whose degree is $[\wt{G}^\ast:\cent{\wt{G}^\ast}{s}]_{\ell'}$. Semisimple characters of $G$ are then the elements of $\Irr(G|\chi_s)$ for the semisimple characters $\chi_s\in\Irr(\wt{G})$  (see \cite[Corollary 2.6.18]{GM20}).
On the other hand, unipotent characters are those in the series corresponding to $s=1$.  The unipotent characters of $\wt{G}$ are trivial on $\zent{\wt{G}}$ and restrict irreducibly to $G$. (See, e.g. \cite[2.3.16 and 4.2.1]{GM20}.)

We begin with the case that the defining characteristic $\ell$ is one of the members of $\pi=\{p,q\}$ in the situation of Question \ref{almostsimpleLie}.

\begin{lem}\label{definingchar}
Let $p$ and $q$ be different primes and let $A$ be as in Question \ref{almostsimpleLie}. Assume that $\mathrm{soc}(A)=S=G/\zent{G}$, where $G=\mathbf{G}^F$ is a group of Lie type with $\mathbf{G}$ of simply connected type defined in characteristic $p$ and that $S$ is not isomorphic to one of the groups considered in Proposition \ref{prop:sporadicexceptschur}.  Then $\irrp{p}{B_p(A)} \cap  \irrp{q}{B_q(A)}\neq1$.
\end{lem}
\begin{proof}
First consider the case $A=S$. Let $G\lhd \wt{G}$ via a regular embedding as above.
By a result originally of Dagger and Humphreys, see \cite[Proposition 1.18 and Theorem 3.3]{Cab18}, $\irr{B_p(S)}$ contains all irreducible characters of $S$ except the Steinberg character, and hence $\irrp{p}{B_p(S)}=\irrp{p}{S}$.  In particular, all semisimple characters of $G$ trivial on the center may be viewed as members of $\irrp{p}{B_p(S)}$.  
  Now, in the proof of \cite[Theorem 3.5]{GSV19}, it is shown that there exists a semisimple character $\chi_t$ of $G$ that is trivial on the center and has $q'$ degree (and $p'$ degree since it is semisimple) 
  and such that $t\in G^\ast$ has $q$-power order.

   Using \cite[Corollary 3.3]{Hiss90}, 
   we see every semisimple character $\wt{\chi}_{\wt{t}}$ of $\wt{G}$ for $\wt{t}\in \wt{G}^\ast$ of $q$-power order lies in $B_q(\wt{G})$.  
   In particular, this implies that every semisimple character of $G$ in a series $\mathcal{E}(G, t)$ for $t$ of $q$-power order also lies in $B_q(G)$. That is, our $\chi_t$ lies in $B_q(S)$. (Indeed, if $t=\iota^\ast(xy)$ with $x\in \wt{G}^\ast$ a $q$-element and $y\in \wt{G}^\ast$ a $q'$-element, note that $y^{|t|}$, and hence $y$, is in $\ker{\iota^\ast}$, since $y$ and $y^{|t|}$ generate the same cyclic subgroup.  Hence, $t=\iota^\ast(\wt{t})$ for $\wt{t}:=x$.)

  Now, assume that $A\neq S$ but that $A=S\cent A P=S\cent A Q$, for every $P\in{\rm Syl}_p(S)$ and $Q\in{\rm Syl}_q(S)$.  
  Then by \cite[Lemma 2.2]{BZ}, we have $|A/S|$ is a power of $p$.  By Theorem \ref{A:Sp'}, this forces $p=2$ to be the defining characteristic and   $A/S$ is an elementary abelian $2$-group.  Note then that  there is a unique $2$-block of $A$ (namely, $B_2(A)$) lying above $B_2(S)$. 
   Let $\chi=\chi_t\in \irrp{2}{B_2(A)} \cap  \irrp{q}{B_q(A)}$ from the previous part.  Then there exists a character $\hat \chi$ in $B_q(A)$ lying above $\chi$, which necessarily is still of $q'$-degree since $|A/S|$ is $q'$. 
   Then $\hat\chi$ lies in $B_q(A)$ and $B_2(A)$, and it now suffices to prove that it has odd degree. 
   
    Note that $A/S$ is generated by graph and field automorphisms (since $\wt{S}/S$ is relatively prime to the defining characteristic) centralizing any Sylow $q$-subgroup $Q$ of $S$ and any Sylow $2$-subgroup $P$ of $S$.  That is, for $\alpha\in A\setminus S$, we have $h^\alpha$ is $S$-conjugate to $h$ for any $2$- or $q$- element $h\in S$.  
    
    First assume that $G\not\in\{\SL_n^\epsilon(2^a), \type{E}_6^\epsilon(2^a)\}$.  Then  $\zent{G}$ is trivial, yielding $G=S$, and further $\wt{G}=G\cong G^\ast$.  Then  $\chi_t^\alpha=\chi_{\alpha^\ast(t)}$ by \cite[Corollary 2.5]{NTT08}, where $\alpha^\ast$ is the corresponding automorphism in $G^\ast\cong G$.  But further, $\alpha^\ast$ centralizes a Sylow $q$-subgroup of $G^\ast$, so $\chi_t^\alpha=\chi_t$ and $\chi=\chi_t$ is $A$-invariant. Then by \cite[Proposition 3.4]{spath12}, $\chi$ extends to $A$, and hence every character of $A$ above $\chi$ is an extension since $A/S$ is abelian.  In particular, $\hat\chi$ is an extension and therefore is of odd degree, as desired.
    
    Now let $G=\type{E}_6^\epsilon(2^a)$ or $G=\SL_n^\epsilon(2^a)$.
Let $\alpha\in A\setminus S$. Now, note that $t$ was constructed to lie in  $[G^\ast, G^\ast]$, so $t$ lies in a Sylow $q$-subgroup of $[G^\ast, G^\ast]\cong S$.  
Then $t^{\alpha^\ast}$ is $G^\ast$-conjugate to $t$.  In particular, $\wt{t}^{\alpha^\ast}$ is $\wt{G}^\ast$-conjugate to $\wt{t}z$ for some $z\in\zent{\wt{G}^\ast}$, where $\wt{t}\in\wt{G}^\ast$ is such that $\iota^\ast(\wt{t})=t$.  
Hence by \cite[Corollary 2.5]{NTT08}, we have $\wt{\chi}_{\wt{t}}^{\alpha}=\wt{\chi}_{\wt{t}z}$, which is another character lying above $\chi_t$, as $\wt{\chi}_{\wt{t}z}=\wt{\chi}_{\wt{t}}\hat z$ for some $\hat z\in \Irr(\wt{G}/G)$ (see \cite[Proposition 2.5.21]{GM20}).  
This implies $\chi_t^\alpha$ is a $\wt{G}$-conjugate of $\chi_t$ for any $\alpha\in A$.

Now, by \cite[Proposition 3.4]{spath12}, there is a character $\chi_0\in\Irr(G|\wt{\chi}_{\wt{t}})$ that extends to $D_{\chi_0}$.  But following the proof there, this $\chi_0$ is the unique character under $\wt{\chi}_{\wt{t}}$ that has multiplicity $\pm1$ in the $D$-invariant virtual character $D_G(\Gamma)$. (Here $D_G(\Gamma)$ is the Alvis--Curtis dual of the Gelfand--Graev character.)  Then $\chi_0^\alpha = \chi_0$ for each $\alpha\in A$, since $\chi_0^\alpha$ must also be multiplicity $\pm1$ in $D_G(\Gamma)$ and, as above, must be a constituent of $\wt{\chi}_{\wt{t}}$ when restricted to $G$.  Hence we see that without loss, we may further assume $\chi_0=\chi_t$ extends to $A$, so $\hat\chi$ is also an extension, since $A/S$ is abelian, completing the proof.
\end{proof}

Note that since characters of $B_q(G)$ and $B_p(G)$ lie in Lusztig series $\mathcal{E}(G, t)$ with $t\in G^\ast$ semisimple $p$-elements, respectively $q$-elements (see \cite[Theorem 9.12]{CE04}), we know that $\irr{B_p(G)} \cap  \irr{B_q(G)}$ contains only unipotent characters.  Further, recall that unipotent characters of $G$ are trivial on $\zent{G}$ and extend to unipotent characters of $\wt{G}$.  Hence for the case $A=S$ in Question \ref{almostsimpleLie}, it suffices to show that there is a unipotent character in $\irr{B_p(\wt{G})} \cap  \irr{B_q(\wt{G})}$ with degree prime to $\{p,q\}$.

For a number $m$ and prime $p$, we write $d_p(m)$ for the order of $m$ modulo $p$ if $p$ is odd, or the order of $m$ modulo $4$ if $p=2$. We next show that Question \ref{almostsimpleLie} holds when $A=S$ is a group of exceptional type in non-defining characteristic.

\begin{lem}\label{lem:except}
Let $S$ be a simple group of exceptional Lie type (including Suzuki and Ree types) defined in characteristic $\ell$, such that $S$ is not isomorphic to one of the groups in Proposition \ref{prop:sporadicexceptschur}. Let $\pi:=\{p,q\}$ for two primes $p, q$ such that $\ell\not\in\pi$.
Then there exists a rational-valued unipotent character lying in $\irr{B_{p}(\wt{S})} \cap  \irr{B_{q}(\wt{S})}$ with $\pi'$-degree that extends to $\aut(S)$. 
\end{lem}

\begin{proof}
For notational convenience, write $p_1:=p$ and $p_2:=q$ and define $d_i:=d_{p_i}(\ell^f)$ for $i=1,2$.  Since the Steinberg character is rational-valued (see \cite{lusztig02}), has degree a power of $\ell$, and extends to $\aut(S)$, we may assume that it does not lie in $\irr{B_{p_1}(\wt{S})} \cap  \irr{B_{p_2}(\wt{S})}$.  By \cite[Main Theorem]{Hiss08} and Proposition \ref{prop:sporadicexceptschur}, we may therefore assume that $S$ is of type  $\tw{2}\type{E}_6$, $\type{E}_6$, $\type{E}_7$, or $\type{E}_8$, and by \cite[Theorems 2.4-2.5]{Malle08}, the unipotent characters we find in $\irr{B_{p_1}(\wt{S})} \cap  \irr{B_{p_2}(\wt{S})}$ will extend to $\aut(S)$.

 Now, if both $d_1$ and $d_2$ 
are regular numbers, then every $p_i'$-degree character lies in $B_{p_i}(\wt{S})$ for $i=1,2$ (see, e.g., the proof of \cite[Lemma 3.6]{RSV21}). In this case, the Steinberg character again satisfies the statement. Hence we may assume that $d_2>2$ is non-regular.  Then we will appeal to the explicit list of unipotent character degrees in \cite[Section 13.9]{carter}, together with the distribution of unipotent characters into blocks found in \cite[Tables 1 \& 2]{BMM93} for non-regular $d_2$.

By \cite[Proposition 5.6]{geck03}, any unipotent character $\chi$ of $\wt{G}$ satisfies $\Q(\chi)=\Q(\psi)$, where $\chi$ lies in the Harish-Chandra series indexed by $(L, \psi)$ for $L$ a split Levi subgroup of $\wt{G}$ and $\psi$ a cuspidal unipotent character of $L$, with some exceptions for $\type{E}_7$ and $\type{E}_8$.  Hence (except for these exceptions), if $\chi$ lies in the principal series or $L$ is of classical type, then it is rational since unipotent characters of classical groups are rational-valued (see \cite{lusztig02}).  By \cite[Tables 1 \& 2]{BMM93}, we see that the only potential non-rational choices then, are in one of the following cases, with notation of \cite{carter}:

\begin{center}
\tiny
\begin{tabular}{|c|c|c|}
\hline
Type of $G$ & $d_2$ & $\chi$\\
\hline
 $\type{E}_7$ & $12$ & $E_6[\theta],\epsilon$ or $E_6[\theta^2],{\epsilon}$ \\
 \hline  
  $\type{E}_8$ & $7$ & $\phi_{4096,11}$ or $\phi_{4096, 12}$ \\
  \hline
  $\type{E}_8$ & $9$ & $\phi_{4096,26}$ or $\phi_{4096, 27}$ \\
  \hline
  $\type{E}_8$ & $9$ & $E_6[\theta],\phi_{1,0}$ or $E_6[\theta^2],\phi_{1,0}$\\
   \hline
   $\type{E}_8$ & $9$ & $E_6[\theta],\phi_{1,3}''$ or $E_6[\theta^2],\phi_{1,3}''$\\
   \hline
   $\type{E}_8$ & $14$ & $E_7[\xi],1$ or $E_7[-\xi],1$\\
   \hline
   $\type{E}_8$ & $18$ & $E_7[\xi],\epsilon$ or $E_7[-\xi],\epsilon$\\
   \hline
   $\type{E}_8$ & $18$ & $E_6[\theta],\phi_{1,0}$ or $E_6[\theta^2],\phi_{1,0}$\\
   \hline
  $\type{E}_8$ & $18$ & $E_6[\theta],\phi_{1,3}''$ or $E_6[\theta^2],\phi_{1,3}''$\\
 \hline
  \end{tabular} 
  \end{center}
  
  \normalsize

Now, using \cite[Section 13.9]{carter} and \cite[Tables 1 \& 2]{BMM93}, we see that
if $d_1$ is also non-regular, then there is a $\pi'$-degree unipotent character in $\irr{B_{p_1}(\wt{S})} \cap  \irr{B_{p_2}(\wt{S})}$ that is not one of the exceptions to rationality in the above table.  Similarly, we see from these sources that if $d_1$ is regular, there is a $\pi'$-degree unipotent character in $\irr{B_{p_2}(\wt{S})}$, and hence in $\irr{B_{p_1}(\wt{S})} \cap  \irr{B_{p_2}(\wt{S})}$, that again does not appear in the above table. 
\end{proof}

We next complete the case $A=S$ in the context of Question \ref{almostsimpleLie} when $S$ is a group of Lie type with $\ell\not\in\pi$ and $2\in\pi$.  To ease our notation, for a fixed power $\ell^f$ of a prime $\ell$, we will define:
\[\Psi_{k}:=(\ell^f)^k-1 \quad\quad\hbox{ and }\quad\quad \Psi_{k}':=(\ell^f)^k+1,\]  so that $\Psi_{2k}=\Psi_k\Psi_k'$.  For any integer $n$ and any prime $p$, we write $(n)_p$  for the $p$-part of the integer.

\begin{prop}\label{prop:nondef}
Let $p\neq \ell$ be distinct odd primes and let $S$ be a simple group of Lie type defined in characteristic $\ell$ that is not isomorphic to one of the groups addressed in Proposition \ref{prop:sporadicexceptschur}.  Write $\pi:=\{2, p\}$.  
Then there exists a rational-valued unipotent character $\chi$ lying in $\irr{B_{2}(\wt{S})} \cap  \irr{B_{p}(\wt{S})}$ with $\pi'$-degree that extends to $\aut(S)$.
\end{prop}

In particular, note that Proposition \ref{prop:nondef} tells us that in the situation there, there exists a rational-valued unipotent character lying in $\irrp{2}{B_{2}({S})} \cap  \irrp{p}{B_{p}({S})}$ that extends to $\aut(S)$.
\begin{proof}[Proof of Proposition \ref{prop:nondef}]
From Lemma \ref{lem:except} and again using \cite{Hiss08} to exclude the cases where the Steinberg character lies in all principal blocks for non-defining primes, we see that we are left to consider the case that $S$ is one of the classical groups:
\begin{itemize}
\item $\PSL_n^\epsilon(\ell^f)=\type{A}_{n-1}(\ell^f)$ if $\epsilon=1$ or $\tw{2}\type{A}_{n-1}(\ell^f)$ if $\epsilon=-1$, with $n\geq 5$
\item $\PSp_{2n}(\ell^f)=\type{C}_n(\ell^f)$ with $n\geq 3$
\item $\POmega_{2n+1}(\ell^f)=\type{B}_n(\ell^f)$ with $n\geq 3$
\item $\POmega_{2n}^+(\ell^f)=\type{D}_n(\ell^f)$ with $n\geq 5$
\item $\POmega_{2n}^-(\ell^f)=\tw{2}\type{D}_n(\ell^f)$ with $n\geq 4$.
\end{itemize}
 Note that every unipotent character is rational-valued by \cite{lusztig02} and lies in $B_2(G)$ and $B_2(\wt{G})$, by \cite[Theorem 21.14]{CE04}. We will provide the details for the case of type $\type{A}_{n-1}$ and note that the degree arguments in the remaining cases are similar, although tedious.

\smallskip

\noindent\textbf{Type $\type{A}_{n-1}$ and $\tw{2}\type{A}_{n-1}$: $S=\PSL_n^\epsilon(\ell^f)$, $n\geq 5$}

In this case, we may write $\wt{G}=\GL_n^\epsilon(\ell^f)$ and $G=\SL_n^\epsilon(\ell^f)$.  Let $e:=d_p(\epsilon \ell^f)$ be the order of $\epsilon \ell^f$ modulo $p$.  Unipotent characters of $\wt{G}$ are parameterized by partitions of $n$. For example, $1_{\wt{G}}$ corresponds to the partition $(n)$ and the Steinberg character corresponds to $(1^n)$.  
Their degrees can be found using the formula in \cite[Section 13.8]{carter}, which we will exploit throughout. 

Writing $n=me+r$ for $r<e$, we have the character $\chi^\lambda$  corresponding to the partition $\lambda$ of $n$ lies in the principal $p$-block if and only if its $e$-core is the partition $(r)$, using the results of Fong and Srinivasan (see \cite[Theorem (5D)]{FS82}). 
Further, since we may assume that the Steinberg does not lie in both principal blocks, we may assume that $r\geq 2$. 

We assume that $\epsilon=1$ and show the details in this case, but analogous calculations work in the case $\epsilon=-1$.

Note that $\Psi_k=\prod_{d|k} \Phi_d$, where $\Phi_d$ denotes the $d$th cyclotomic polynomial in $\ell^f$.  We also remark that $\Phi_d(\ell^f)$ is divisible by $p$ if and only if $d=p^x d_p(\ell^f)$ for some nonnegative integer $x$ and $p^2$ can only divide $\Phi_d(\ell^f)$  if $d=d_p(\ell^f)$. 
Further, $\Phi_d(\ell^f)$ is divisible by $2$ if and only if $d$ is a power of $2$.  We will use this to determine a unipotent character $\chi^\lambda$ lying in $\Irr_{\pi'}(B_{2}(G)\cap B_{p}(G))$.

Let  $n=2^{a_1}+2^{a_2}+\ldots+2^{a_t}$ with $a_1<a_2<\ldots<a_t$ be the $2$-adic expansion of $n$.  Let $1\leq t_0\leq t$ be the smallest such that $r<2^{a_{t_0}}$, and write $T:=2^{a_t}+2^{a_{t-1}}+\cdots +2^{a_{t_0}}$ and $a:=n-T+1$.  Note that if $t_0\neq 1$, we have $r\geq 2^{t_0-1}$.

\medskip

\emph{Case I:} Suppose first that $r=n-T=a-1$.   
In this case, consider $\lambda= (1^{me}, r)$.  Then $\chi^\lambda$ lies in $B_p(G)$, and the $\pi$-part of $\chi^\lambda(1)$ is the $\pi$-part of 
 
\begin{equation}\label{eq:IIdeg}
\frac{\prod_{k=r}^{n-1}\Psi_k}{\prod_{k=1}^{me}\Psi_k}=\frac{\prod_{k=me+1}^{n-1}\Psi_k}{\prod_{k=1}^{r-1}\Psi_k},
\end{equation}

which is certainly $p'$.  Now, we see this is further odd, since $me=T$ and $(T+i)_2=(i)_2$ for $1\leq i<n-T$.

\emph{Case II:} Now we suppose that $r\neq (a-1)=n-T= 2^{a_1}+\cdots +2^{a_{t_0-1}}$.  Then $T\neq me$ and we have $n-2^{a_{t_0-1}}\geq me>n-2^{a_{t_0}}=2^{a_1}+\cdots+2^{a_{t_0-1}}+2^{a_{t_0+1}}+\cdots+2^{a_t}$.  
Then we have $T-me=2^{a_{t_0}}-2^{a_{t_0-1}}-\cdots - 2^{a_1}-x$ for some $0<x<2^{a_{t_0}}$.  
Then writing $R=2^{a_{t_0}}-T+me$, we see $R< 2^{a_{t_0}+1}$ and hence must satisfy $(R)_2\leq 2^{a_{t_0}}$.  Then we have $(me)_2=(R)_2=|T-me|_2$.   
That is, we see $|T-me|$ has same $2$-part as $me$. We consider the cases $r\geq a$ and $r<a-1$ separately.

\emph{Case IIa:} Here assume that $r\geq a$.  In this case, consider the partition $\lambda=(1^{n-r-a}, a, r)=(1^{me-a}, a, r)$.  The corresponding unipotent character $\chi^\lambda$ lies in $B_{p}(G)$ and the $\pi$-part of $\chi^\lambda(1)$ is the $\pi$-part of 
\begin{equation}\label{eq:firstdeg}
\frac{\Psi_{r-a+1}\prod_{k=n-a+2}^n \Psi_k \prod_{k=n-r-a+1}^{n-a}\Psi_k}{\Psi_{n-r}\prod_{k=1}^{r}\Psi_k\prod_{k=1}^{a-1}\Psi_k}.
\end{equation}

Notice that $me=n-r\leq n-a$ since $a\leq r$.  Then the product $\Psi_{r-a+1}\prod_{k=n-a+2}^n \Psi_k$ is relatively prime to $p$, since no $k$ for which $\Psi_k$ is involved in this product is divisible by $e$.  Further, the $p$-part of $\prod_{k=n-r-a+1}^{n-a}\Psi_k=\prod_{k=me-a+1}^{n-a}\Psi_k$ is the $p$-part of $\Psi_{me}$, which also appears in the denominator, and hence we see $\chi^\lambda(1)$ is relatively prime to $p$.

To see that $\chi^\lambda(1)$ is odd, it is convenient to rewrite \eqref{eq:firstdeg} as 
\[\frac{\Psi_{T-me}\prod_{k=T+1}^n \Psi_k \prod_{k=T-r}^{T-1}\Psi_k}{\Psi_{me}\prod_{k=1}^{r}\Psi_k\prod_{k=1}^{n-T}\Psi_k}.\]

Then $\frac{\prod_{k=T+1}^n \Psi_k}{\prod_{k=1}^{n-T}\Psi_k}$ is odd because the $2$-part of $T+i$ for $1\leq i\leq n-T$ is the same as the $2$-part of $i$.  
Similarly, $\frac{\Psi_{r-n+T} \prod_{k=T-r}^{T-1}\Psi_k}{\Psi_{n-r}\prod_{k=1}^{r}\Psi_k}$ is odd since the $2$-part of $T-i$ is the same as the $2$-part of $i$ for $1\leq i\leq r$ since $r<2^{t_0}$.  (And, recall as discussed above that 
the $2$-part of $r-n+T=T-me$ is the $2$-part of $n-r=me$.)

\smallskip

\emph{Case IIb:} Now suppose that $r<a-1$ and that either $e>a-1$ or $p\nmid (m-1)$. Note that if $e\leq a-1$, we must still have $2e>a-1$. (Indeed, note that $2e>2r\geq 2\cdot 2^{a_{t_0-1}}=2^{a_{t_0-1}+1}>2^{a_1}+\cdots 2^{a_{t_0-1}}=a-1$. )  For similar reasons, note that in this case $2e-r> a-1-r\geq e-r$ but that $r+e\neq a-1$, so $a-r-1\neq e$ and $a-r-1$ cannot be a multiple of $e$.

 Here consider $\lambda=(1^{me-a}, r+1, a-1)$.  Then again $\chi^\lambda$ is in $B_p(G)$, and the $\pi$-part of $\chi^\lambda(1)$ is the $\pi$-part of the same expression as \eqref{eq:firstdeg} with the $\Psi_{r-a+1}=\Psi_{T-me}$ term replaced with $\Psi_{me-T}$.  Hence we again see that $\chi^\lambda(1)\in\Irr_{2'}(G)$.  Further, we see that if $e>a-1$, this is still $p'$ by the same arguments as before.  Now, if $e\leq a-1$, note that the $p$-part of $\chi^\lambda(1)$ comes from $\frac{\Psi_{me}\Psi_{(m-1)e}}{\Psi_{me}\Psi_e}$, since $2e>a-1$, and hence is $p'$ since we have assumed $p\nmid (m-1)$.

\emph{Case IIc:} Now suppose that $r<a-1$, $e\leq a-1$, and $p\mid (m-1)$.  As above, note that $2e>a-1$.  In this case, let $1< s_0\leq t$ be the smallest such that $e+r< 2^{s_0}$ and write $S:=2^{a_t}+\cdots+2^{s_0}$ and $b:=n-S+1$. (Note that our assumptions $r<e<n-T$ force $t_0\leq s_0\leq t_0+1$.)  Here we argue similar to before, taking $\lambda=(1^{me-e}, r+e)$ if $e+r=n-S=b-1$; $\lambda=(1^{n-r-e-b}, b, r+e)$ if $e+r\geq b$; and $\lambda=(1^{n-r-e-b}, r+e+1, b-1)$ if $e+r<b-1$, and arrive at similar degree formulas (with the roles of $\{me, r, a, T\}$ played by $\{me-e, r+e, b, S\}$) that are $p'$ because of our assumption that $p\mid(m-1)$ (and hence $p\nmid m, m-2, m-3$) and are odd for the same reasons as cases I, IIa, IIb, respectively.

\medskip

\medskip

Finally, we remark that the exact same partitions work in the case of $\epsilon=-1$, noting that here $\Psi_k$ can be replaced by $(\ell^f)^k-(-1)^k$ in the degree formulas.  This completes the proof for types $\type{A}$ and $\tw{2}\type{A}$.

\bigskip

\noindent\textbf{Types $\type{B}_n$ and $\type{C}_n$ with $n\geq 3$}

Let $S$ be $\type{B}_n(\ell^f)$ or $\type{C}_n(\ell^f)$ with $n\geq 3$.  
In this case, the unipotent characters of $S$, $G$, or $\wt{G}$ are parametrized by so-called symbols of rank $n$ with odd defect.  A symbol of rank $n$ is a pair of partitions $ {\lambda_1\hbox{ } \lambda_2\hbox{ } \cdots\hbox{ } \lambda_a} \choose {\mu_1\hbox{ } \mu_2\hbox{ } \cdots\hbox{ } \mu_b}$ $ = {\lambda\choose\mu}$, where $\lambda_1<\lambda_2<\cdots<\lambda_a$, $\mu_1<\mu_2<\cdots<\mu_b$, $\lambda_1$ and $\mu_1$ are not both $0$, and $n=\sum_i \lambda_i+\sum_j \mu_j - \lfloor{\left(\frac{a+b-1}{2}\right)^2}\rfloor$.  (The symbol $\lambda\choose\mu$ is equivalent to $\mu\choose\lambda$, and if $\lambda_1$ and $\mu_1$ are both $0$, the symbol is equivalent to $ \lambda_2-1\hbox{ } \cdots\hbox{ } \lambda_a-1 \choose  \mu_2-1\hbox{ } \cdots\hbox{ } \mu_b-1$.)  The \emph{defect} of a symbol is $|b-a|$. Given an integer $e$, an $e$-hook is a pair of non-negative integers $(x, y)$ with $y-x=e$, $x\not\in\lambda$ (resp. $\mu$), and $y\in \lambda$ (resp. $\mu$).  The $e$-core of a symbol is obtained by successively removing $e$-hooks, which means replacing $y$ by $x$ in $\lambda$ (resp. $\mu$) and then replacing the result with an equivalent symbol satisfying that $\lambda_1$ and $\mu_1$ are not both $0$. An $e$-cohook is defined similarly, except that $x\not\in\lambda$ and $y\in\mu$ (or $x\not\in\mu$ and $y\in\lambda$), and the $e$-cocore is obtained by removing $e$-cohooks, which means removing $y$ from $\mu$ and adding $x$ to $\lambda$ (resp. removing $y$ from $\lambda$ and adding $x$ to $\mu$), and again replacing the result with an equivalent symbol satisfying that $\lambda_1$ and $\mu_1$ are not both $0$.

In this case, we let $e$ be the order of $({\ell^f})^2$ modulo $p$.  Then two symbols are in the same $p$-block if and only if they have the same $e$-core, respectively $e$-cocore, if $p\mid \Psi_e$, respectively $p\mid \Psi_e'$.   The trivial character is represented by the symbol $n \choose \emptyset$, which has $e$-core and $e$-cocore $r\choose\emptyset$, where $0\leq r< e$ is the remainder when $n:=me+r$ is divided by $e$.  Note that again every unipotent character  lies in the principal $2$-block.

As before, write $n=2^{a_1}+2^{a_2}+\ldots+2^{a_t}$ with $a_1<a_2<\ldots<a_t$ for the $2$-adic expansion of $n$, let $1\leq t_0\leq t$ be the smallest such that $r<2^{a_{t_0}}$, and define $T:=2^{a_t}+2^{a_{t-1}}+\cdots +2^{a_{t_0}}$ and $a:=n-T+1$.

From here, the considerations are very similar to the case of type $\type{A}$ and $\tw{2}\type{A}$ above, using the degree formulas in \cite[Section 13.8]{carter}.  Namely, we again have the following cases.

\indent Case I: $r=n-T=a-1$ so $me=T$.

\indent Case IIa: $r\geq a$.

\indent Case IIb: $r<a-1$ and either $e>a-1$ or $p\nmid (m-1)$.

\indent Case IIc: $r<a-1, e\leq a-1$, and $p\mid (m-1)$.

 Table \ref{tab:typeC} lists symbols that complete the proof in cases I and IIa,b. Note that in Case I, the $\pi$-part of the corresponding character degree is 
 \[\frac{\prod_{k=me+1}^n \Psi_{2k} \prod_{k=me-r}^{me-1}\Psi_{2k}}{\prod_{k=1}^{r}\Psi_{2k}\prod_{k=1}^{r}\Psi_{2k}}=\frac{\prod_{k=T+1}^n \Psi_{2k} \prod_{k=T-(n-T)}^{T-1}\Psi_{2k}}{\prod_{k=1}^{n-T}\Psi_{2k}\prod_{k=1}^{n-T}\Psi_{2k}},\]
 which is $\pi'$.  Also note that in Case I when $r=0$, the listed character is the Steinberg character.  In Cases IIa-b, we have the $\pi$-part of the corresponding character degree is the same as that of \eqref{eq:firstdeg}, with each $\Psi_k$ replaced with $\Psi_{2k}$ and $T-me$ replaced by $|T-me|$, multiplied by $\frac{\Psi'_{me}\Psi'_T}{2\Psi'_{|T-me|}}$ in the case $p\mid \Psi_e$ or $m$ even and by $\frac{\Psi_{me}\Psi'_T}{2\Psi_{|T-me|}}$ when $m$ is odd and $p\mid \Psi_e'$.
  Then since $T$ is even (so $(\Psi_T')_2=2$) and the $2$-part of $|T-me|$ is again the same as that of $me$, we see that these are still odd.  Further, note that $p\mid \Psi_{me}$ in the first situation and $p\mid \Psi_{me}'$ in the latter, so these are $\pi'$ since further $me\neq T$ and $T=me\pm x$ where $r=n-T\mp x$ (recall that even if $r<a-1$, we still have $x=a-r-1$ is not a multiple of $e$), we see $e\nmid T$ so $p\nmid \Psi_T'$. 
 
 In Case IIc, we use the same strategy as in the case of type $\type{A}$; namely, using analogous symbols with the roles of $\{me, r, a, T\}$ played by $\{me-e, r+e, b, S\}$, where $b$ and $S$ are as defined there.

\begin{table}\small
\caption{Symbols for Unipotent Characters in $\irrp{2}{B_{2}(\wt{S})}\cap \irrp{p}{B_{p}(\wt{S})}$ for types $\type{B}_n$ and $\type{C}_n$ with $n\geq 3$, where $n=me+r$ }\label{tab:typeC}
\begin{tabular}{?c?c?c?c?}
\Xhline{3\arrayrulewidth}
Condition on $p$ & $p$ divides $\Psi_e$ & $p$ divides $\Psi'_e$, $m$ even & $p$ divides $\Psi_e'$, $m$ odd \\
\Xhline{3\arrayrulewidth}

Case I & \multicolumn{3}{c?}{${0, 1, 2, \ldots, me-r-1, me}\choose{ 1, 2, \ldots, me-r-1, me}$}\\
\hline
Case IIa  & \multicolumn{2}{c?}{${1, 2, \ldots, me-a, me, T}\choose{0, 1, 2, \ldots, me-a}$} & {${0, 1, 2, \ldots, me-a, me}\choose{ 1, 2, \ldots, me-a, T}$} \\
\hline
Case IIb  & \multicolumn{2}{c?}{{${1, 2, \ldots, me-a, T, me}\choose{0, 1, 2, \ldots, me-a}$}} & {${0, 1, 2, \ldots, me-a, me}\choose{ 1, 2, \ldots, me-a, T}$} \\
\hline

\Xhline{3\arrayrulewidth} 
\end{tabular}
\end{table}

\bigskip

\noindent\textbf{Types $\type{D}_n$ and $\tw{2}\type{D}_n$ with $n\geq 4$}

In this case the unipotent characters of $S, G,$ and $\wt{G}$ are in bijection with symbols of rank $n$ and satisfying $|b-a|$ is $0\pmod 4$ (unless $\lambda=\mu$, in which case there are two characters) for type $\type{D}$ and $2\pmod 4$ for type $\tw{2}\type{D}$.  Let $e$ be defined the same way as in type $\type{B}$ and $\type{C}$. The block distribution is described the same way as for those types.  
We again write $n=me+r$ with $r<e$.

For type $\type{D}_n$, the trivial character is represented by the symbol $n \choose 0$, which has $e$-core $r\choose 0 $ if $e\nmid n$ and $\emptyset \choose \emptyset$ if $e\mid n$.  It has $e$-cocore $r\choose 0$ if $m$ is even and $e\nmid n$; $0 \hbox{ } r \choose \emptyset$ if $m$ is odd and $e\nmid n$; $\emptyset \choose \emptyset$ if $m$ is even and $e\mid n$; and $e\choose 0$ if $m$ is odd and $e\mid n$.

If $r=1$ or if $e\mid n$ and $p\mid \Psi_e$ or $m$ is even, then the Steinberg character satisfies our conditions.  So, we assume that either $r>1$, in which case we may consider the same cases I, IIa-c as before, or $r=0, p\mid \Psi_e'$, and $m$ is odd.  Table \ref{tab:typeD} lists symbols that complete the proof in cases I, IIa-b using arguments like before, and a similar strategy to before works for IIc with the roles of $\{me, r, a, T\}$ again played by $\{me-e, r+e, b, S\}$.   In the case $r=0, p\mid \Psi_e'$, and $m$ is odd, cases and symbols analogous to those used when $m$ is even can be used, with the roles of $\{r, me\}$ replaced by $\{e, n-e\}$ (that is, the same symbols as IIc but with $r=0$), sometimes with some small modification.  For example, when $e=n-S$, analogous to case I, the symbol ${n-e}\choose e$ works if $p\nmid (m-2)$, and ${1, n-e}\choose {0, e+1}$ works otherwise.

For type $\tw{2}\type{D}_n$, the trivial character is represented by the symbol $0\hbox{ }  n \choose \emptyset$, which has $e$-core $0 \hbox{ }  r \choose \emptyset$ when $e\nmid n$ and $0 \hbox{ } e \choose \emptyset$ if $e\mid n$.  The $e$-cocore is $0 \hbox{ }  r \choose \emptyset$ if $e\nmid n$ and $m$ is even, $r\choose 0$ if $e\nmid n$ and $m$ is odd, $e\choose 0$ if $e\mid n$ and $m$ is even, and $\emptyset\choose\emptyset$ if $e\mid n$ and $m$ is odd.

When $r=1$ or if $r=0, p\mid \Psi_e'$, and $m$ is odd, the Steinberg character again works.  When $r\geq 2$, we have the same cases as the previous situations. Table  \ref{tab:type2D} lists symbols that complete the proof for I and IIa-b, and again the same strategy works for IIc. In the remaining case, the situation is again similar, with the role of $\{r, me\}$ replaced by $\{e, n-e\}$.  
\begin{table}\small
\caption{Symbols for Unipotent Characters in $\irrp{2}{B_{2}(\wt{S})}\cap \irrp{p}{B_{p}(\wt S)}$ for type $\type{D}_n$ with $n\geq 5$, where $n=me+r$ }\label{tab:typeD} 
\begin{tabular}{?c?c?c?c?}
\Xhline{3\arrayrulewidth}
Condition on $p$ & $p$ divides $\Psi_e$ & $p$ divides $\Psi'_e$, $m$ even & $p$ divides $\Psi_e'$, $m$ odd \\
\Xhline{3\arrayrulewidth}

Case I & \multicolumn{3}{c?}{$me \choose r$}\\
\hline
Case IIa  & \multicolumn{2}{c?}{${1, 2, \ldots, me-a, me, T}\choose{0, 1, 2, \ldots, me-a+1}$} & {${0, 1, 2, \ldots, me-a, me}\choose{ 1, 2, \ldots, me-a+1, T}$} \\
\hline
Case IIb  & \multicolumn{2}{c?}{{${1, 2, \ldots, me-a, T, me}\choose{0, 1, 2, \ldots, me-a+1}$}} & {${0, 1, 2, \ldots, me-a, me}\choose{ 1, 2, \ldots, me-a+1, T}$} \\
\hline

\Xhline{3\arrayrulewidth} 
\end{tabular}

\end{table}
\begin{table}\small
\caption{Symbols for Unipotent Characters in $\irrp{2}{B_{2}(\wt S)}\cap \irrp{p}{B_{p}(\wt S)}$ for type $\tw{2}\type{D}_n$ with $n\geq 4$, where $n=me+r$ }\label{tab:type2D}
\begin{tabular}{|c?c?c?c|}
\hline
Condition on $p$ & $p$ divides $\Psi_e$ & $p$ divides $\Psi'_e$, $m$ even & $p$ divides $\Psi_e'$, $m$ odd \\
\hline
Case I  & \multicolumn{3}{c|}{$r, me \choose \emptyset$}\\
\hline
Case IIa, $n\neq T$ &\multicolumn{2}{c?}{{${1, 2, \ldots, me-a, me-a+1, me, T}\choose{0, 1, 2, \ldots, me-a}$}} & {${1, 2, \ldots, me-a, T}\choose{0, 1, 2, \ldots, me-a+1, me}$}\\
\hline
Case IIb, $n\neq T$ &\multicolumn{2}{c?}{{${1, 2, \ldots, me-a, me-a+1, T, me}\choose{0, 1, 2, \ldots, me-a}$}} & {${1, 2, \ldots, me-a, T}\choose{0, 1, 2, \ldots, me-a+1, me}$}\\
\hline
Case IIa,b, $n= T$ &\multicolumn{3}{c|}{{${1, 2, \ldots, me-1, me}\choose{0, 1, 2, \ldots, me, T}$}}  \\

\hline

\end{tabular}

\end{table}

\bigskip

Finally, note that the characters constructed here further extend to $\aut(S)$ by \cite[Theorems 2.4-2.5]{Malle08}.
\end{proof}

\begin{rem}\label{rem:oddp1}
Although we do not need it here, we remark briefly that a similar strategy can likely be used in the case of the classical groups when $\pi=\{p_1,p_2\}$ is a pair of odd primes.  Letting $e_i:=d_{p_i}(\epsilon\ell^f)$ in the case $\PSL_n^\epsilon(\ell^f)$ and $d_{p_i}((\ell^f)^2)$ in the remaining cases, we may write $n=m_ie_i+r_i$ and assume without loss that $r_1\leq r_2$.
  By \cite[Theorem 1.9]{MO83}, there is a height-preserving bijection between $B_{p_i}(\wt{G})$ and $B_{p_i}(\GL^\epsilon_{m_ie_i}(q))$ in the case of $\PSL_n^\epsilon(\ell^f)$, and a similar statement holds in the remaining cases using \cite[Theorem 5.17]{malle19}. (Although one must sometimes be careful here and work with the general or special orthogonal groups rather than $\wt{G}$.) 
  Using this, we may assume that $r_1=0$ and replace $r_2$ with $r_2-r_1$, so that $n=m_1e_1=m_2e_2+r_2$ with $e_2>r_2\geq 0$.  
From here, writing the $p_1$-adic expansion of $m_1$ as $m_1=c_1p_1^{a_1}+c_2p_1^{a_2}+\ldots+c_tp_1^{a_t}$ with $a_1<\ldots<a_t$ and $1\leq c_i<p_1$ for $1\leq i\leq t$, a similar strategy to Proposition \ref{prop:nondef} can be used.  
\end{rem}

Note that the Galois group $\gal:=\mathrm{Gal}(\Q^{\mathrm{ab}}/\Q)$ acts on $\irrp{p}{B_p(G)}$ for any finite group $G$ and prime $p\mid|G|$. That is, $\chi^\sigma\in\irrp{p}{B_p(G)}$ for any $\chi\in\irrp{p}{B_p(G)}$ and $\sigma\in\gal$. We will use this to help identify which extension of the character from Proposition \ref{prop:nondef} lies in the principal block of $A$.

\begin{lem}\label{lem:rationalext}
Suppose $X\lhd A$ are finite groups and $p$ is a prime dividing $|X|$ but not dividing $|A/X|$.  Suppose further that $A=X\cent{A}{P}$ for $P\in \syl{p}{X}$ and that $A/X$ is abelian.  Let $\chi\in\irr{B_p(X)}$ be rational-valued and extend to $A$.  Then there is a unique extension of $\chi$ to $A$ lying in $B_p(A)$, and this extension is rational-valued.
\end{lem}

\begin{proof} 
By Theorem \ref{isomblocks}, there is a unique extension $\hat\chi$ of $\chi$ to $A$ lying in $B_p(A)$. Then since $\chi^\sigma=\chi$, for any $\sigma\in\gal$, this means $\hat\chi^\sigma\in\irr{B_p(A)}$  lies above $\chi$ as well, forcing $\hat\chi^\sigma=\hat\chi$.
\end{proof}

We are now ready to answer Question \ref{almostsimpleLie} for groups of Lie type when $2\in\pi$.

\begin{prop}\label{lietype}
Let $S=G/\zent{G}$ be a simple group of Lie type not listed in Proposition \ref{prop:sporadicexceptschur}, let $p$ be an odd prime dividing $|S|$, and let $S\leq A\leq \aut(S)$ satisfying the hypothesis of Question \ref{almostsimpleLie}.  Then the set $\irrp{2}{B_2(A)}\cap\irrp{p}{B_p(A)}$ is nontrivial. In particular, Question \ref{almostsimpleLie} holds when $\pi=\{2,p\}$ and $S=G/\zent{G}$ is a simple group of Lie type. 
\end{prop}

\begin{proof}
By Lemma \ref{definingchar}, we may assume that the defining characteristic of $G$ is not in $\pi$.  Then by Proposition \ref{prop:nondef}, there is a rational-valued unipotent character $\chi\in\Irr_{\pi'}(B_2(\wt S)\cap B_p(\wt S))$ that extends to $\aut(S)$.  

Write $C$ for the centralizer in $\aut(S)$ of a Sylow $2$-subgroup of $P$, and let $C=C_1C_2$ be as in Theorem \ref{A:Sp'}(2). By \cite[Lemma 3.3]{GMS22}, there exist extensions $\chi_2$ and $\chi_p$ of $\chi$ to $B_2(\wt{S}C)$ and to $B_p(\wt{S}C)$, respectively.   
Now, notice that $\wt{S}C_1/\wt{S}$ has odd order and is abelian, and that $(\wt{S}C_1\cap A)/\wt{S}$ further has order prime to $p$, by Lemma \ref{A:Sp'}(1). (Indeed, note that the odd part of $\Out(S)/\wt{S}$ is abelian - see, e.g. \cite[Theorems 2.5.12-2.5.14]{GLS3} for information about the structure of $\out(S)$.) Note that the restriction of $\chi_2$ and of $\chi_p$ to $A_1:=\wt{S}C_1\cap A$ must lie in the respective principal blocks.  Then by Lemma \ref{lem:rationalext}, these restrictions are the unique extension of $\chi$ to $B_2(A_1)$ and $B_p(A_1)$, and are rational-valued.  But since $(A_1)/\wt{S}$ is abelian and odd, we see only one extension to $A_1$ is rational-valued, and hence these two characters are the same.  That is, there is an extension $\hat\chi$ of $\chi$ to $B_2(A_1)\cap B_p(A_1)$. Since $A/A_1$ is a $2$-group, we know $B_2(A)$ is the unique block of $A$ lying above $B_2(A_1)$. In particular, the restriction of $\chi_p$ to $A$, which extends $\hat\chi$, must lie in $B_p(A)\cap B_2(A)$.
\end{proof}

\begin{rem}\label{rem:oddp1complication}
We end this section by discussing where our techniques fail when $2\not\in\pi$. 
Using Theorem \ref{A:Sp'}, in the context of Question \ref{almostsimpleLie}, we have $|A/S|$ is $\pi'$ when $2\not\in\pi$. Suppose that we have a nontrivial rational character in $\irrp{p}{B_p(\wt{S})}\cap\irrp{q}{B_q(\wt{S})}$ extending to $\aut(S)$ as before, and that   $|A\wt{S}/\wt{S}|$ is abelian.  Then using Lemma \ref{lem:rationalext}, there are again unique extensions $\hat\chi_p$ and $\hat\chi_q$ of $\chi$ in $B_p(A)$ and $B_q(A)$, respectively.  So, we would need to argue that $\hat\chi_p=\hat\chi_q$.  Arguing like before, note that $\hat\chi_p=\hat\chi_q\beta$ for some linear character $\beta$ satisfying $\beta^2=1$.  However, in this case, we may have $A\wt{S}/\wt{S}$ is even (for example, $A$ may contain a field automorphism of order 2), and hence it is plausible that $\beta\neq 1$.

We thank P. H. Tiep for pointing out the following explicit example of the above complication. Let $S=\alt_8=\SL_4(2)$ and $A=\aut(S)=\sym_8$, and let $\pi=\{3,5\}$. Then letting $\chi\in\Irr(S)$ be either the (unique) character of degree 14 or the (unique) character of degree 64, we have $\chi\in\Irr_{3'}(B_3(S))\cap \Irr_{5'}(B_5(S))$. Of the two extensions of $\chi$ to $A$,  one lies in $B_3(A)$ but not $B_5(A)$, and the other lies in $B_5(A)$ but not $B_3(A)$.  (However, it is worth pointing out that there is an extension of the character of degree 56 that does lie in $B_5(A)\cap B_3(A)$, so $A$ still satisfies Question \ref{questionalmostsimple} for $\pi$.)
\end{rem}

\subsection{Alternating Groups}

Similar calculations to those used in Proposition \ref{prop:nondef} above for type $\PSL_n(\ell^f)$ 
 also work for the case of the alternating groups $\alt_n$.  Recall that for $n\geq 8$, $\aut(\alt_n)=\sym_n$, the symmetric group on $n$ letters.  For $\pi=\{2,p\}$, any $\chi\in\irrp{\pi}{\sym_n}$ in both principal blocks will restrict irreducibly to a $\pi'$ character also in both principal blocks of $\alt_n$.

The elements of $\Irr(\sym_n)$ are indexed by partitions of $n$, and two characters lie in the same $2$-block, respectively $p$-block, if the corresponding partitions have the same $2$-core, respectively $p$-core.  The trivial character is represented by $(n)$, so that a character is in the principal $p$-block if and only if its $p$-core is $(r)$, where $n=mp+r$ with $r<p$, and in the principal $2$-block if the $2$-core is $\emptyset$ for $n$ even and $(1)$ for $n$ odd.  The degree of the characters are given by the hooklength formula.

\begin{prop}\label{prop:alts}
Let $p$ be an odd prime and let $n\geq 8$.  Then there is a nontrivial element in $\irrp{2}{B_2(\sym_n)}\cap\irrp{p}{B_p(\sym_n)}$.  In particular, Question \ref{almostsimpleLie} holds for $q=2$ and $p$ odd when $S$ is an alternating group $\alt_n$ with $n\geq 8$.
\end{prop}
\begin{proof}
  Write $n=mp+r$ with $r<p$ and $n=2w+r_0$ with $r_0\in\{0,1\}$. By \cite[Theorem 1.10]{MO83}, there are height-preserving bijections between $B_2(\sym_n)$ and $B_2(\sym_{2w})$, and between $B_p(\sym_n)$ and $B_p(\sym_{mp})$.  For this reason, we may assume that either $r=0$ or that $r\geq 1$ and $n$ is even. 

Assume first that $r\geq 1$ and $n$ is even.  As in Proposition \ref{prop:nondef}, write  $n=2^{a_1}+2^{a_2}+\ldots+2^{a_t}$ with $a_1<a_2<\ldots<a_t$ be the $2$-adic expansion of $n$; let $1\leq t_0\leq t$ be the smallest such that $r<2^{a_{t_0}}$; and write $T:=2^{a_t}+2^{a_{t-1}}+\cdots +2^{a_{t_0}}$ and $a:=n-T+1$.  
Here we may consider the same cases as in the proof of Proposition \ref{prop:nondef} for type $\type{A}_{n-1}$, and use the same partitions, with the role of $e$ there now played by $p$.  We summarize this in Table \ref{tab:alt}.  Here in the cases of IIc, $S$ and $b$ are defined similar to before: $s_0$ is the smallest such that $p+r<2^{a_{s_0}}$; $S:=2^{a_t}+\cdots 2^{a_{s_0}}$; and $b:=n-S+1$.

  We note that the degrees in this case are analogous to the degrees in the case of $\type{A}_{n-1}$, now with each $\Psi_k$ replaced by the integer $k$ and $e$ replaced by $p$.  For example, in case I, the degree is \[\frac{\prod_{k=mp+1}^{n-1}k}{\prod_{k=1}^{r-1}k},\] which is \eqref{eq:IIdeg} with each $\Psi_k$ replaced by the integer $k$ and $e$ replaced by $p$.
Then in each case, the degrees are $\pi'$ using the same arguments as in the proof of Proposition \ref{prop:nondef}.

\begin{table}\small
\caption{Partitions for Characters in $\irrp{2}{B_{2}(\sym_n)}\cap \irrp{p}{B_{p}(\sym_n)}$ for $n\geq 8$, where $n$ is even and $n=mp+r$ with $r<p$.}\label{tab:alt}
\begin{tabular}{|c?c|}
\hline
Condition on $p$ & partition $\lambda$ \\
\hline
Case I: $r=n-T=a-1$ &   
$ (1^{mp}, r)$   \\
\hline
Case IIa: $r\geq a$ & $(1^{n-r-a}, a, r)=(1^{mp-a}, a, r)$ \\
\hline
Case IIb: $r< a-1$ and & \\either $e>a-1$ or $p\nmid (m-1)$ & $(1^{mp-a}, r+1, a-1)$\\
\hline
Case IIc: $p\leq a-1$ and $p\mid (m-1)$ & $(1^{mp-p}, r+p)$ if $p+r=n-S$\\
&  $(1^{n-r-p-b}, b, r+p)$ if $p+r\geq b$\\ &  $(1^{n-r-p-b}, r+p+1, b-1)$ if $p+r<b-1$\\
\hline
\end{tabular}
\end{table}

Now assume $r=0$, so that $n=mp=2w+r_0$ with $r_0\in\{0,1\}$.  Here  the same partitions as in Case IIc in Table \ref{tab:alt}, but taking $r=0$, give members of $\irrp{2}{B_2(\sym_n)}\cap\irrp{p}{B_p(\sym_n)}$ except in the case $p\geq b$ and $p\mid m$ or $p<b-1$ and $p\mid m$ or $m-2$. Note that if $S=n$, the second line is $(1^{mp-p}, p)$, the same as the first, and the character is still $\pi'$. Hence in the remaining cases $p\neq n-S$ when $p\mid m$ or $m-2$, we may also assume that $S\neq n$. Then $(1^S, n-S)$ gives a character in $B_2(\sym_n)\cap B_p(\sym_n)$ with degree $\frac{\prod_{k=S+1}^{n-1}k}{\prod_{k=1}^{n-S-1}k}$, which is odd using the same reasoning as before and is $p'$ since $p\nmid (m-1)$ and $n-2p=(m-2)p$ must be less than $S$. 
\end{proof}

\section{Groups which are $p$ and $q$-solvable}\label{sec:pqsolvable}
In this section we prove Conjectures B and C under solvability conditions.
 
\begin{lem}\label{ber}
Suppose that $\irrp{p}G=\irrp{p} {G/N}$.  Then $N$ has a normal $p$-complement.
\end{lem}
\medskip

\begin{proof}
This follows from a theorem of Y. Berkovich. (See,
for instance,  \cite[Theorem 7.7]{Na18}.)
\end{proof}

\medskip

\begin{thm}\label{nt}
Suppose $p$ and $q$ are different primes.
Then $\irrp {p}G=\irrp{q} G$ if and only if there are $P \in \syl pG$ and
$Q \in \syl qG$ such that $[P,Q]=1$, $PQ$ is abelian, and $\norm GP=\norm GQ$.
\end{thm}
\medskip

\begin{proof}
This is the content of \cite{NT}. For $p$-solvable and $q$-solvable groups, this follows from the main result of \cite{NWo}. See also \cite{NRS2}. 
\end{proof}

\medskip
The following is Conjecture B for $p$- and $q$-solvable finite groups.

\begin{thm}
 Let $p$
and $q$ be different primes. Assume that $G$ is $p$-solvable
and $q$-solvable.
Suppose that $$\irrp{p}{B_p(G)}=\irrp{q}{B_q(G)} \, .$$ Then $p$ and $q$
do  not divide $|G|$.
\end{thm}

\medskip

\begin{proof} Using Lemma \ref{th}, it is enough to show that $pq$ does not divide $|G|$.
Let $L= \oh{p'} G$, $M=\oh{q'} G$
and  $K=\oh{p'}G\oh{q'}G$. By \cite[Theorem 10.20]{nbook}, we know that 
$\irr{{B_p}(G)}=\irr{G/L}$ and $\irr{B_q(G)}=\irr{G/M}$.

Let $X=M \cap L$.  Since $\oh{p'}{G/X}=L/X$, we can assume that $X=1$.

By hypothesis, we have that $\irrp{p}{G/L} \sbs \irr{G/M}$.
Therefore ${\rm Irr}_{p'}(G/L)={\rm Irr}_{p'}(G/K)$.
By Lemma \ref{ber}, this implies that $K/L$ (and therefore $M$) has a normal $p$-complement. Thus $M=\oh p G$.   By the same argument, $L=\oh qG$.  

Suppose that $K \sbs Y \nor G$ and $G/Y$ is $p'$ and $q'$.  Then $\oh{p'}Y=L$
and $\oh{q'}Y=M$. Let $\tau \in \irrp{p}{Y/L}$, and let $\chi \in \irr G$ be over $\tau$. Then $\chi$ has $p'$-degree so it lies in $\irrp{q}{G/M}$. Thus $\tau \in \irrp{q}{Y/M}$. By the same
argument with the primes reversed and using induction, 
we conclude that $Y$ is not divisible by $pq$, and we are done in this case.

As it is obvious from the hypothesis, notice 
that $\irrp{p}{G/K}=\irrp{q}{G/K}$.   Let $\pi=\{p,q \}$.  By Theorem \ref{nt}, we have that $G/K$ has
abelian Hall $\pi$-subgroups. By the Hall-Higman Lemma 1.2.3 (see \cite[Theorem 3.21]{Is2}), and using the previous paragraph,
if $R/K= \oh{\pi'}{G/K}$, we have that $G/R$ is an abelian $\pi$-group.

Write $G/R=P_0/R \times Q_0/R$, where $P_0/R$ is an abelian $p$-group and $Q_0/R$ is
an abelian $q$-group.  Now, let $\gamma \in {\rm Irr}_{p'}(P_0/L)$. Let $\psi \in \irr G$ be over $\gamma$. Since $G/P_0$ is a $q$-group,
it follows that $\psi$ has $p'$-degree.
 We conclude that  $\psi$ has $q'$-degree, by using the hypothesis. Therefore, 
$\gamma=\psi_{P_0}$ is $Q_0$-invariant.  Now, notice that
$P_0/L$ is a $q'$-group.
Let $Q \in \syl qG$. 
Then $Q/L$ acts coprimely on $P_0/L$ fixing all the $p'$-degree irreducible
characters. Let $P_1/L\in{\rm Syl}_{p}(P_0/L)$ be $Q$-invariant (see \cite[Theorem 3.23]{Is2}). By \cite[Corollary B]{MN11} we have that   $[P_1/L,Q/L]=1$.
In particular,  $Q/L \sbs \cent{G/L}{U/L}$, where $U/L=\oh p{G/L}$.
Again, by the Hall-Higman Lemma 1.2.3, 
we conclude that $Q=L$ and $Q_0=R$. By the same reason $P_0=R$,
and we deduce that $G=R$. By the fourth paragraph,
we have that $K=G$.  If $L>1$, let $1 \ne \lambda \in \irr{L}$ linear
and consider $\chi=\lambda \times 1_M$. Then $\chi \in \irrp{q}{G/M}$
but does not contain $L$ in its kernel. Thus $L=1$ and by the same
reason $M=1$.  
\end{proof}

\medskip

It seems difficult to
characterize when 
  $\irrp{p}{B_p(G)} \sbs \irrp{q}{B_q(G)}$ group theoretically, as shown by
   $G=A6.2_3$, $p=2$ and $q=5$. Certainly, this is an interesting problem.
   
   \medskip
   The following is the $p$-$q$-solvable case of Conjecture C.
   \medskip
   
   \begin{thm}
   Suppose that $G$ is a finite $p$-solvable and $q$-solvable group,
   for different primes $p$ and $q$. Then $q$ does not divide $\chi(1)$ for all
   $\chi \in {\rm Irr}_{p'}(B_p(G))$ and $p$ does not divide $\psi(1)$
   for all  
   $\psi \in {\rm Irr}_{q'}(B_q(G))$ if and only if there are $P \in \syl pG$ and $Q \in \syl qG$
   such that $[P,Q]=1$.
    \end{thm}
    
    \begin{proof} $L=\oh{p'}G$ and $M=\oh{q'}G$ and notice that ${\rm Irr}(B_p(G))={\rm Irr}(G/L)$ and ${\rm Irr}(B_q(G))={\rm Irr}(G/M)$ by \cite[Theorem 10.20]{nbook}.
    Suppose that $[P,Q]=1$. Then $Q$ centralizes $K/L$, where  $K/L=\oh p{G/L}$. By the Hall--Higman Lemma 1.2.3, we have that $QL/L$ is contained in $K/L$.
    Hence, $Q \sbs L$ and therefore $q$ does not divide the irreducible character degrees
    of $G/L$. Reasoning analogously we obtain that $p$ does not divide the irreducible character degrees
    of $G/M$.
    
    Suppose now that $q$ does not divide $\chi(1)$ for all
   $\chi \in {\rm Irr}_{p'}(B_p(G))$ and $p$ does not divide $\psi(1)$
   for all  
   $\psi \in {\rm Irr}_{q'}(B_q(G))$. We prove this implication by induction on $|G|$. If $X=L\cap M>1$, by induction $G/X$ has nilpotent Hall $\{p,q\}$-subgroups and using the Schur--Zassenhaus theorem we conclude that $G$ has nilpotent Hall $\{p,q\}$-subgroups. Thus we may assume that $X=1$.
   
   We have that ${\rm Irr}_{p'}(G/L) \sbs {\rm Irr}_{q'}(G/L)$. By the main result of
   \cite{NWo}, there is a Sylow $p$-subgroup $P$ of $G$ and a Sylow $q$-subgroup
   $Q$ of $G$, such that $\norm GP \sbs \norm GQ L$. 
   If $K/L=\oh p{G/L}$, then we have that $K/L$ normalizes $QL/L$.
   Since $QL/L$ normalizes $K/L$ and have coprime orders, we deduce that $QL/L$
   centralizes $K/L$. By the Hall--Higman Lemma 1.2.3, we deduce that $Q\sbs L$.
   By the same reason, we have that $P \sbs M$. We conclude that $[P,Q]\subseteq[M,L]=1$ as wanted. 
      \end{proof}
   
   \medskip
   We have mentioned in the Introduction that our conjectures seem to admit even a further generalization
   using Galois automorphisms. Indeed, if $\sigma_p$ is the automorphism of $\Q^{\rm ab}$ that
  fixes $p'$-roots of unity and sends $p$-power roots of unity $\zeta$ to $\zeta^{1+p}$,
  it should be possible to replace ${\rm Irr}_{p'}(B_p(G))$ by the subset of its
  fixed elements under $\sigma_p$. We note that our work in Section \ref{sectionsimples} supports this generalization in the case of Theorem A.  Namely, the characters in $\Irr_{p'}(B_p(A))\cap\Irr_{q'}(B_q(A))$ constructed in Lemma \ref{definingchar} and Propositions \ref{prop:nondef} and \ref{prop:alts} are also fixed by $\sigma_p$ and $\sigma_q$.  Indeed, this is clear for the rational characters in Propositions \ref{prop:nondef} and \ref{prop:alts}, but the characters described in Lemma \ref{definingchar} for $S$ can also be chosen to be fixed by $\sigma_p$ and $\sigma_q$ (see  \cite[Proposition 4.5 and its proof]{GHSV}), so that this is also true for the characters of $A$ by again applying Theorem \ref{isomblocks}.


\begin{thebibliography}{99}



\bibitem[Al76]{Alp76}
{\sc J. L. Alperin},
\newblock Isomorphic blocks.
\newblock {\em J. Algebra} \textbf{43} (1976), 694--698.


\bibitem[BNOT07]{BNOT}
{\sc C. Bessenrodt, G.  Navarro,  J. B. Olsson,
Pham H. Tiep}, On the Navarro-Willems conjecture for blocks of finite groups. 
\emph{J. Pure Appl. Algebra \bf208} (2007),   481--484. 


\bibitem[BZ08]{BZ}
{\sc C. Bessenrodt, J. Zhang}, Block separations and inclusions, 
\emph{Adv. Math. \bf 218} (2008) 
485--495.


\bibitem[BMM93]{BMM93}
{\sc M. Brou\'{e}, G. Malle, and J. Michel}, Generic blocks of
finite reductive groups, \emph{Ast\'{e}risque} \textbf{212} (1993),
7--92.

\bibitem[Cab18]{Cab18}
{\sc M. Cabanes}, Local methods for blocks of finite simple groups.
Local Representation Theory and Simple Groups, 179--265, \emph{EMS
Ser. Lect. Math., Eur. Math. Soc.}, Z{\"u}rich, 2018.

\bibitem[CE04]{CE04}
{\sc M. Cabanes and M. Enguehard}, \emph{Representation theory of
finite reductive groups}, New Mathematical Monographs \textbf{1},
Cambridge University Press, Cambridge, 2004.

 \bibitem[Car85]{carter}
{\sc R.\,W. Carter}, \emph{Finite groups of Lie type. Conjugacy
classes and complex characters}, Wiley and Sons, New York et al,
1985, 544 pp.

\bibitem[Da77]{Dad77}
{\sc E. C. Dade},
\newblock Remarks on isomorphic blocks.
\newblock {\em J. Algebra} \textbf{45} (1977), 254--258.

\bibitem[FS82]{FS82}
{\sc P.~Fong and B.~Srinivasan}, The blocks of finite general linear and unitary groups, \emph{Invent. Math.} \textbf{69} (1982), 109-153.

\bibitem[GAP]{GAP}
{\sc The GAP group}, `{\it {\sf GAP} - groups, algorithms, and
programming}', Version 4.10.0, 
2018,
{\sf http://www.gap-system.org}.

\bibitem[Ge03]{geck03}
{\sc M.~Geck}, Character values, Schur indices and character sheaves. 
\emph{Represent. Theory} \textbf{7} (2003), 19--55.

\bibitem[GM20]{GM20}
{\sc M.~Geck and G.~Malle}. (2020). \emph{The Character Theory of Finite Groups of Lie Type: A Guided Tour} (Cambridge Studies in Advanced Mathematics).  Cambridge University Press. 

\bibitem[GHSV21]{GHSV}
{\sc E.~Giannelli, N.~Hung,  A.~A.~Schaeffer~Fry, and C.~Vallejo}, Characters of $\pi '$-degree and small cyclotomic fields, \emph{Ann. di Mat. Pura ed Appl. (1923 -)}  \textbf{200} (2021), 1055--1073.

\bibitem[GMS22]{GMS22}
{\sc E.~Giannelli, J.~M.~Mart{\'i}nez, and A.~A.~Schaeffer~Fry}, Character degrees in blocks and defect groups, \emph{J. Algebra} (to appear). {arXiv:2110.13227}



\bibitem[GSV19]{GSV19}
{\sc E.~Giannelli, A.~A.~Schaeffer~Fry, and  C.~Vallejo}, Characters of $\pi'$-degree, \emph{Proc. Amer. Math. Soc.} \textbf{147}:11 (2019),  4697--4712.

\bibitem[Gl68]{Glauberman}
{\sc G.~Glauberman}, Weakly closed elements of Sylow subgroups, \emph{Math. Z.} \textbf{107} (1968), 1--20.

\bibitem[GLS94]{GLS3}
{\sc D. Gorenstein, R. Lyons, and R. Solomon}, \emph{The
Classification of the Finite Simple Groups}, Math. Surveys Monogr.,
vol. 3, Amer. Math. Soc., Providence, 1994.


\bibitem[Gr82]{Gross82}
{\sc F.~Gross}, Automorphisms which centralize a Sylow $p$-subgroup, \emph{J. Algebra} \textbf{77} (1982), 202--233. 

\bibitem[Hi90]{Hiss90}
{\sc G.~Hiss}, Regular and semisimple blocks of finite reductive groups, \emph{J. London Math. Soc.} \textbf{41} (1990), 63--68.

\bibitem[Hi08]{Hiss08}
{\sc G.~Hiss}, Principal blocks and the Steinberg character, \emph{Algebra Colloq.} \textbf{17}:03 (2010), 361--364.

 \bibitem[IS76]{IS}
 {\sc I. M. Isaacs and S. D. Smith,} `{\it A note on groups of $p$-length 1}', J. Algebra, {\bf 38} (1976),   531--535.
 
  \bibitem[Is08]{Is2}
 {\sc I. M. Isaacs}  Finite group theory.  Graduate Studies in Mathematics, 92. American Mathematical Society, Providence, RI, 2008.

 

\bibitem[LWXZ20]{LWXZ}
{\sc Y. Liu, W. Willems,
H. Xiong, J. Zhang},
 Trivial intersection of blocks and nilpotent subgroups,
\emph{J. Algebra \bf 559} (2020) 510--528.


\bibitem[LWXZ21]{LWXZ2}
{\sc Y. Liu, W. Willems,
H. Xiong, J. Zhang},
 Trivial intersection of blocks and nilpotent subgroups, Addendum
\emph{J. Algebra \bf 584} (2021) 161--162.



\bibitem[LWWZ22]{LWWZ}
{\sc Y. Liu, W. Willems, L. Wang and J. Zhang} 
Brauer's height zero conjecture for two primes holds true,
\emph{preprint}, 


\bibitem[Lu02]{lusztig02}
{\sc G. Lusztig}, Rationality properties of unipotent representations. \emph{J. Algebra} \textbf{258} (2002), 1--22.

\bibitem[Ma08]{Malle08}
{\sc G. Malle}, Extensions of unipotent characters and the inductive
McKay condition, \emph{J. Algebra} \textbf{320} (2008), 2963--2980.

\bibitem[Ma20]{malle19}
{\sc G.~Malle}.
\newblock On the number of characters in blocks of quasi-simple groups, \emph{Algebr. Represent. Theor.} \textbf{23} (2020), 513--539.

\bibitem[MN11]{MN11}
{\sc G. Malle, G. Navarro}, Extending Characters from Hall Subgroups,
\emph{Doc. Math. \bf 16} (2011), 901-919. 

\bibitem[MN20]{MN20}
{\sc G. Malle, G. Navarro}, Brauer's height zero conjecture
for two primes, 
\emph{Math. Z. \bf 295} (2020), no. 3-4, 1723--1732. 

\bibitem[MO83]{MO83}
{\sc G.\,O. Michler and J.\,B. Olsson}, Character correspondences in
finite general linear, unitary and symmetric groups, \emph{Math. Z.}
\textbf{184}, 203--233.

\bibitem[Na98]{nbook}{\sc  G. Navarro}, \emph{Characters   and Blocks   of Finite Groups}. LMS Lecture Note Series 250,  Cambridge University Press, 1998. 
  

\bibitem[Na18]{Na18}
{\sc G. Navarro},
Character Theory and the McKay Conjecture,
Cambridge University Press, 2018. 

\bibitem[NRS22]{NRS2}
{\sc G. Navarro, N. Rizo, A. A. Schaeffer Fry},
Principal blocks for different primes II, preprint.

\bibitem[NT12]{NT12}
{\sc G. Navarro,  P. H. Tiep,}  Brauer's height zero conjecture for the 2-blocks of maximal defect,
 J. Reine Angew. Math. {\bf 669} (2012), 225--247. 

\bibitem[NT22]{NT}
{\sc G. Navarro, P. H. Tiep},
Characters and sets of primes, preprint.

\bibitem[NTT08]{NTT08}
{\sc G. Navarro, P. H. Tiep, and A. Turull}, Brauer characters with cyclotomic field of values,
\emph{J. Pure App. Algebra} \textbf{212} (2008), 628--635.

\bibitem[NWi97]{NW}
{\sc G. Navarro, W. Willems}
 When is a $p$-block a $q$-block? 
 \emph{Proc. Amer. Math. Soc. \bf 125} (1997), no. 6, 1589--1591.
 
 \bibitem[NWo98]{NWo}
{\sc G. Navarro, T. R. Wolf}
 Thomas Character degrees and local subgroups of $\pi$-separable groups.
 \emph{ Proc. Amer. Math. Soc. \bf126} (1998),   2599--2605. 
 
 
 \bibitem[RV11]{RV}
{\sc D. O. Revin, E. P. Vdovin}
An existence criterion for Hall subgroups
of finite groups
 \emph{J. Group Theory \bf 14} (2011), 93--101.
 
 \bibitem[RSV21]{RSV21}
{\sc N.~Rizo, A.~A.~Schaeffer Fry, and C.~Vallejo}, Principal blocks with 5 irreducible characters. \emph{J. Algebra} \textbf{585} (2021), 316--337.

\bibitem[Sp12]{spath12}
{\sc B.~Sp{\"a}th}, Inductive McKay condition in defining characteristic.
\emph{Bull. London Math. Soc.} \textbf{44} (2012), 426--438.

\bibitem[Wi86]{W86}
 {\sc W. Willems}, $p^*$-Theory and modular representation theory, 
 \emph{J. Algebra \bf104} (1986), 135--140.

  \end{thebibliography}
\end{document}